\pdfoutput=1
\documentclass[10pt]{article}
\newcommand{\sign}{\mathop{sgn}}
\newcommand{\abs}[1]{\left|#1 \right|}

\newcommand{\N}{\mathbb{N}}
\newcommand{\Z}{\mathbb{Z}}
\newcommand{\R}{\mathbb{R}}
\newcommand{\C}{\mathbb{C}}
\newcommand{\D}{\mathbb{D}}
\newcommand{\Dbar}{\overline{\mathbb{D}}}
\newcommand{\indf}[1]{\mathbf{1}_{#1}}
\newcommand{\Nsup}[1]{\left\|#1\right\|_{\sup}}
\newcommand{\Linfty}[1]{L^{\infty}\left( #1 \right)}
\newcommand{\Ninfty}[1]{\left\|#1 \right\|_{\infty}}

\newcommand{\None}[1]{\left|#1 \right|_{1}}
\newcommand{\Ltwo}[1]{L^{2}\left( #1 \right)}
\newcommand{\Ntwo}[1]{\left|#1 \right|_{2}}
\newcommand{\Bs}{\mathcal{S}}
\newcommand{\Ns}[1]{\left\|#1 \right\|_{s}}
\newcommand{\MBs}[1]{Lip_{s}\left(#1 \right)}
\newcommand{\NBs}[1]{\left\|#1 \right\|_{\Bs}}
\newcommand{\Bw}{\mathcal{W}}
\newcommand{\NBw}[1]{\left\|#1 \right\|_{\Bw}}
\newcommand{\ULip}{\mathbf{L}}
\newcommand{\MUL}[1]{Lip_{u}\left(#1\right)}
\newcommand{\NUL}[1]{\left\|#1\right\|_{\ULip}}
\newcommand{\Tl}{\mathcal{L}}
\newcommand{\NTl}[1]{\left\|#1\right\|_{\Tl}}
\newcommand{\Ta}{\mathcal{H}}
\newcommand{\NTa}[1]{\left\|#1\right\|_{\Ta}}
\newcommand{\image}[1]{\text{image}\left( #1 \right)}
\newcommand{\Leb}{Leb}
\newcommand{\floor}[1]{\left\lfloor #1 \right\rfloor}
\usepackage{amsmath}
\usepackage{amssymb}
\usepackage{amsthm}
\usepackage{hyperref}
\usepackage{cleveref}
\usepackage{authblk}
\usepackage{color}
\usepackage{graphics}
\usepackage{graphicx}
\usepackage{enumerate}
\numberwithin{equation}{section}

\theoremstyle{plain}
\newtheorem{thm}{Theorem}[section]
\newtheorem{cor}{Corollary}[thm]
\newtheorem{lem}[thm]{Lemma}
\newtheorem{prop}[thm]{Proposition}

\theoremstyle{definition}
\newtheorem{dfn}[thm]{Definition}

\theoremstyle{remark}

\crefname{dfn}{Definition}{Definitions}
\crefname{rmk}{Remark}{Remarks}
\crefname{ex}{Example}{Examples}   
\crefname{thm}{Theorem}{Theorems}
\crefname{prop}{Proposition}{Propositions}
\crefname{lem}{Lemma}{Lemmas}
\crefname{slem}{Sublemma}{Sublemmas}
\crefname{cor}{Corollary}{Corollaries}
\crefname{corl}{Corollary}{Corollaries}
\crefname{apdx}{Appendix}{Appendices}
\crefname{chapter}{Chapter}{Chapters}
\newenvironment{subproof}[1][\proofname]{%
  \begin{proof}[#1]%
  }{%
  \end{proof}%
}
\title{Limit Theorems for Generalized Baker's Transformations}
\author{Seth W. Chart\\ {\small schart@uvic.ca}}
\affil{
  Department of Mathematics\\
  Towson University\\
  78000 York Road\\
  Room 316\\
  Towson, MD, 21252
}
\date{\today}
\begin{document}
\maketitle
%<% Abstract
\begin{abstract}
%<% Paragraph 1
In this paper we study decay of correlations and limit theorems for generalized baker's transformations \cite{Bose&Murray2013,Bose1989,AlexanderYorke1984,Tsujii2001,Rams2003}. Our examples are piecewise non-uniformly hyperbolic maps on the unit square that posses two spatially separated lines of indifferent fixed points.\par
%>%
%<% Paragraph 2
We obtain sharp rates of mixing for Lipschitz functions on the unit square and limit theorems for H\"older observables on the unit square. Some of our limit theorems exhibit convergence to non-normal stable distributions for H\"older observables. We observe that stable distributions with any skewness parameter in the allowable range of $[-1,1]$ can be obtained as a limit and derive an explicit relationship between the skewness parameter and the values of the H\"older observable along the lines of indifferent fixed points.\par

%>%
%<% Paragraph 3
This paper is the first application of anisotropic Banach space methods \cite{Blank&Keller&Liverani2002,Baladi&Tsujii2007,Demers&Liverani2008} and operator renewal theory \cite{Sarig2002,Gouezel2004} to generalized baker's transformations. Our decay of correlations results recover the results of \cite{Bose&Murray2013}. Our results on limit theorems are new for generalized baker's transformations.
%>%
\end{abstract}
%>%
%<% Section: Introduction
\section{Introduction}
\label{sec:Introduction}
%<% Paragraph 1
Intermittent baker's transformations (IBTs) are invertible, non-uniformly hyperbolic, and area preserving skew products on the unit square that generalize the classical baker's transformation \cite{Bose&Murray2013,Bose1989,AlexanderYorke1984,Tsujii2001,Rams2003}.\par

%>%
%<% Paragraph 2
If a map $T\colon X \to X$ preserves a probability measure $\mu$, $\psi\colon X \to \R$ is in $L^{\infty}(\mu)$, and $\eta\colon X \to \R$ is in $L^{1}(\mu)$, then we define the correlation function by
\[
  Cor(k;\psi,\eta,T) = \abs{\int \psi \circ T^{k} \,\eta\, d\mu - \int\psi \,d\mu\int\eta\,d\mu}.
\]
If the limit of the correlation function as $k$ tends to infinity is zero for all $\psi \in L^{\infty}$ and $\eta \in L^{1}$, then the map is strongly mixing. If $Cor\left( k;\psi,\eta,T \right) = O\left( \frac{1}{k^{\nu}} \right)$ for some $\nu>0$, then we say that the correlations decay at a polynomial rate. If the rate is independent of the choice of $\psi$ and $\eta$ in some class of functions, then we say that $T$ displays a polynomial rate of decay of correlations for observables in that class. If the class contains functions $\psi$ and $\eta$ such that\footnote{The notation $f(k)\asymp g(k)$ as $k \to \infty$ indicates that $f$ and $g$ are in bounded ratio for $k$ sufficiently large. See \Cref{sec:Escape From Fixed Points} a discussion of asymptotics.} $Cor(k;\psi,\eta,T) \asymp \frac{1}{k^{\nu}}$ as $k \to \infty$, then we say that the rate is sharp. A limit theorem is a statement of the form: If (H) and $\int \psi \, dm =0$, then
\begin{align}
  \label{eqn:Generic Limit Theorem}
  \frac{1}{A_{n}}\sum_{k=0}^{n-1} \psi \circ T^{k} 
  &\xrightarrow{dist} 
  Z, \quad \text{as } n\to \infty.
\end{align}
Where $(H)$ is a dynamical hypothesis, $A_{n}$ is a sequence of real numbers, and $Z$ is a real valued random variable. It is well known \cite{Liverani1996} that if a map displays a summable rate of decay of correlations and mild additional hypotheses (H), then (\ref{eqn:Generic Limit Theorem}) is satisfied with $A_{n} = \sqrt{n}$ and $Z=N(0,\sigma)$ a normal distribution with variance determined by the correlation function. When a map displays a rate of decay of correlations that is not summable it is possible \cite{Gouezel2004-Intermittent} to prove that (\ref{eqn:Generic Limit Theorem}) is satisfied with a different normalizing sequence and $Z$ a stable law, which may not be normal. In this case more delicate hypotheses are required.\par

%>%
%<% Paragraph 3
In \cite{Bose&Murray2013} the authors prove that every IBT displays a sharp polynomial rate of decay of correlations for H\"older observables via the Young tower method \cite{Young1998}. The Young tower method relies on analyzing an expanding factor map of the hyperbolic map in question and obtaining rates of decay of correlations for the factor map. These rates are then lifted to the full hyperbolic map via \emph{a posteriori} arguments. Operator renewal theory \cite{Sarig2002,Gouezel2004,Gouezel2005} has been used to obtain sharp polynomial rates of decay of correlation estimates and convergence to stable laws when the rate of decay of correlations is not summable. The anisotropic Banach space methods of \cite{Blank&Keller&Liverani2002,Baladi&Tsujii2007,Demers&Liverani2008} are used to analyze the transfer operator associated to multidimensional maps directly without the need to pass to one dimensional factors.\par

%>%
%<% Paragraph 4
In this paper we construct anisotropic Banach spaces adapted to IBTs modeled on the work of \cite{Demers&Liverani2008,Liverani&Terhesiu2015}. This allows us to analyze the transfer operator of the two dimensional piecewise non-uniformly hyperbolic IBT directly. IBTs posses lines of indifferent fixed points that obstruct exponential rates of decay of correlations and the Lasota-Yorke type arguments used to obtain such results. In order to treat indifferent fixed points for the full two dimensional map and obtain sharp polynomial rates of decay of correlations we apply operator renewal theory. We also use the operator renewal method to obtain limit theorems for both the summable and non-summable rates of decay of correlations.\par

%>%
%<% Paragraph 5
Non-normal stable distributions posses a skewness parameter that ranges in $[-1,1]$. In most dynamical applications limit theorems exhibit convergence to a stable distribution with skewness parameter either equal to $1$ or $-1$. In this paper we obtain limit theorems that exhibit convergence to stable distributions with any skewness parameter in $[-1,1]$ and derive an explicit relationship between this parameter and properties of the IBT and the observable in question. We also obtain convergence to the normal distribution with both standard and non-standard normalizing sequences.\par

%>%
%<% Paragraph 6
We will obtain the spectral decomposition required to apply operator renewal theory in \Cref{sec:Main Results}. In \Cref{sec:Rates} we recover the sharp polynomial rates of decay of correlations for Lipschitz functions. In \Cref{sec:Limit Theorems} we obtain limit theorems for IBTs, which is a new result. See \Cref{sec:Statement of Results} for statements of the theorems.

%>%
%<% Subsection: Statement of Results
\subsection{Statement of results}
\label{sec:Statement of Results}
%<% Paragraph 1
A function $\phi \colon [0,1] \to [0,1]$ is an intermittent cut function (ICF) if it is smooth, strictly decreasing, and there exist constants $\alpha_{0},\alpha_{1}>0$, $c_{0},c_{1}>0$, and differentiable functions $h_{0}$ and $h_{1}$ defined on a neighborhood of zero with $h_{j}(0)=0$ and $Dh_{j}(x) = o\left( x^{\alpha_{j}-1} \right)$, such that 
\begin{align}
  \label{eqn:Cut Function Expansion Left}
  1-\phi(x) & = c_{0}x^{\alpha_{0}} + h_{0}(x),\\
  \label{eqn:Cut Function Expansion Right}
  \phi(1-x) & = c_{1}x^{\alpha_{1}} + h_{1}(x).
\end{align} 
Every IBT is uniquely determined by an ICF. We refer to the constants $c_{j}$ and $\alpha_{j}$ above as the \emph{contact coefficients} and \emph{contact exponents} of $B$ respectively.\par
\begin{figure}[!h]
  \begin{center}
    \def\svgwidth{0.40\textwidth}
    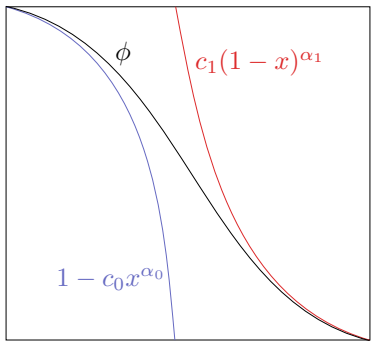
    \caption{An intermittent cut function.}
    \label{fig:Cut Function}
  \end{center}
\end{figure}

%>%
%<% Paragraph 2
Given an IBT $B$ we will induce on a subset $\Lambda$ of the unit square and apply operator renewal theory to obtain the following.
  \begin{thm}
  \label{thm:Main Correlation}
    Suppose that $B\colon [0,1]^2 \to  [0,1]^2$ is an Intermittent Baker's Transformation, as defined in \Cref{sec:Maps}, with contact exponents $\alpha_{j}>0$. Let $\alpha = \max\left\{ \alpha_{0},\alpha_{1} \right\}$. If $\eta$ and $\psi$ are Lipschitz functions on $\Lambda$, then $Cor(k;\psi,\eta,B)= O\left( k^{-\frac{1}{\alpha}} \right)$. If additionally $\int \eta \, d \Leb \neq 0$ and $\int \psi\, d\Leb \neq 0$, then $Cor(k;\psi,\eta,B) \asymp k^{-\frac{1}{\alpha}}$.
  \end{thm}
It is important to note that we obtain a sharp decay rate in \Cref{thm:Main Correlation}. If $\eta$ and $\psi$ are supported on $\Lambda$, $\int_{\Lambda}\eta \neq 0$, and $\int_{\Lambda} \psi \neq 0$, then \Cref{eqn:Mean Zero Identity} shows that the rate of decay of correlation is asymptotically in bounded ratio with $n^{-\frac{1}{\alpha}}$.\par

%>%
%<% Paragraph 3
The following is a collection of limit theorems for IBTs. See \Cref{thm:Main Limit Theorem} for precise statements.
\begin{thm}
\label{thm:Main Limit Laws}
Suppose that $\psi\colon \left[ 0,1 \right]^{2} \to \R$ is $\gamma$-H\"older for some $\gamma\in(0,1]$ and $\int_{[0,1]^{2}} \psi \,d \Leb = 0$. Let $M_{0} =\int_{0}^{1} \psi(0,y^{1+\frac{1}{\alpha_{0}}})\,dy$ and $M_{1} =\int_{0}^{1} \psi(1,y^{1+\frac{1}{\alpha_{1}}})\,dy$.
\begin{enumerate}[i.]
    \item If\footnote{This hypotheses is weakened substantially in \Cref{sec:Limit Theorems}.} $\alpha_{0},\alpha_{1}<1$, then (\ref{eqn:Generic Limit Theorem}) is satisfied with $A_{n} = \sqrt{n}$ and $Z = N(0,\sigma^2)$ where $\sigma^2$ depends on $Cor(k;\psi,\psi,T)$ for all $k\ge 0$.
    \item If $\alpha_{0} > \alpha_{1}$, $\alpha_{0}>1$, and $M_{0} >0$, then (\ref{eqn:Generic Limit Theorem}) is satisfied with $A_{n} = n^{\frac{\alpha_{0}}{\alpha_{0}+1}}$ and $Z$ a stable law of index $1+\tfrac{1}{\alpha_{0}}$, and skewness parameter $1$.
    \item If $\alpha_{0} = \alpha_{1}=:\alpha$, $\alpha>1$, $M_{0}>0$ and $M_{1}<0$, then (\ref{eqn:Generic Limit Theorem}) is satisfied with $A_{n} =n^{\frac{\alpha}{\alpha+1}}$ and $Z$ a stable law of index $1+\tfrac{1}{\alpha}$, and skewness parameter determined by $M_{0}$ and $M_{1}$. Any skewness parameter in $[-1,1]$ is attainable.
    \item If $\alpha_{0} = \alpha_{1} = 1$, $M_{0} \neq 0$, and $M_{1} \neq 0$, then (\ref{eqn:Generic Limit Theorem}) is satisfied with $A_{n} =\sqrt{n\log(n)}$ and $Z=N(0,\sigma^2)$ where $\sigma^{2}$ is determined by $M_{0}$ and $M_{1}$.
  \end{enumerate}
\end{thm}
%>%
%>%
%>%
%<% Section: Maps
\section{Maps}
\label{sec:Maps}
%<% Paragraph 1
Generalized baker's transformations are area preserving maps of the unit square that generalize the classical baker's transformation. Roughly speaking a generalized baker's transformation is a map that realizes the following procedure. First, select a function $\phi\colon[0,1] \to [0,1]$ and let $A= \int \phi$. Second, slice the unit square along the line $x=A$. Third, press the left portion of the square under the graph of $\phi$. Fourth, press the right portion of the square over the graph of $\phi$. If the pressing is done so that area is preserved and every vertical line is mapped affinely to a vertical line, then this procedure determines a map $B\colon [0,1]^2 \to [0,1]^2$.\par

  %<% Figure: Intermittent baker's transformation
  \begin{figure}[h!]
    \begin{center}
      \def\svgwidth{0.80\textwidth}
      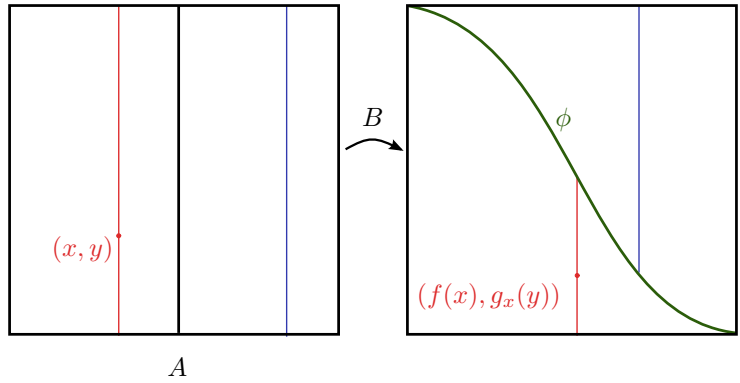
      \caption{An intermittent baker's transformation.}
      \label{fig:Intermittent Baker's Transformation}
    \end{center}
  \end{figure}
  %>%
%>%
%<% Paragraph 2
We will make the rough description of the last paragraph percise in the case that the function $\phi$ is an ICF as defined in \Cref{sec:Statement of Results}. As before let $A = \int \phi$ denote the area of the region below the graph of $\phi$. The associated IBT $B$ can be defined in terms of an \emph{expanding factor map} $f\colon[0,1] \to [0,1]$ and \emph{fibre maps} $g_{x}\colon[0,1] \to [0,1]$, by the formula
%<% Displayed Equation: Skew product representation of a IBT.  
\begin{align}
    \label{eqn:Skew Product}
    B(x,y)
    &=
    \left( f(x),g_{x}(y) \right).
  \end{align}
%>%
We define $f$ in \Cref{sec:Factor Map} below and note that the fibre maps are defined for each $x \in [0,1]$ by
  %<% Displayed Equation: Definition of fiber maps.
  \begin{align}
    \label{eqn:Fiber Map}
    g_{x}(y) 
    &=
    \left\{\begin{array}[h]{ll}
      \phi\left(f(x)\right) y, & \text{ if } x \in [0,A);\\
        \left[1-\phi\left(f(x)\right)  \right] y + \phi\left( f(x) \right), & \text{ if } x \in [A,1].\\
      \end{array}
      \right.
    \end{align}

  %>%
%>%
%<% Paragraph 3
For convenience we introduce the following notation for iterates of $B$,
  %<% Displayed Equation: Skew product notation for iterates.
  \begin{align}
    g^{(0)}_{x}(y) &= y \notag;\\
    \label{eqn:Fiber Map Iterates Left}
    g^{(n+1)}_{x}(y) &= g_{f^{n}(x)}\left(g^{(n)}_{x}(y)\right), \; n\ge0;\\
    \label{eqn:Fiber Map Iterates Right}
    B^{n}(x,y) &= \left( f^{n}(x), g^{(n)}_{x}(y) \right).
  \end{align}
%>%
%>%
%<% Subsection: The expanding factor map.
\subsection{Expanding Factor}
\label{sec:Factor Map}
%<% Figure: Expanding factor of an IBT
\begin{figure}[!ht]
\begin{center}
  \def\svgwidth{0.40\textwidth}
  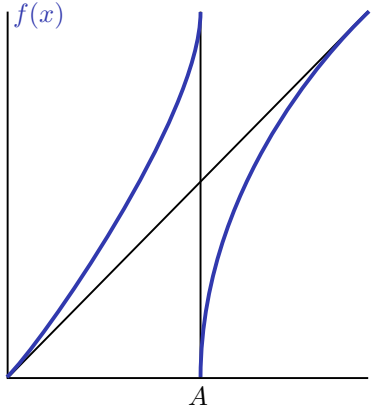
  \caption{The expanding factor of an IBT.}
  \label{fig:Factor Map}
\end{center}
\end{figure}
%>%
%<% Paragraph 1
We define $w_{0} \colon [0,1] \to [0,A]$ and $w_{1} \colon [0,1] \to [A,1]$ by
%<% Displayed Equation: Inverse branches of the expanding factor.
\begin{align}
\label{eqn:Factor Map Branch Left}
w_{0}(x) &= \int_{0}^{x} \phi(t) \, dt,\\
\label{eqn:Factor Map Branch Right}
w_{1}(x) &= A+\int_{0}^{x} 1-\phi(t) \, dt.
\end{align}
%>%
\par
%>%
%<% Paragraph 2
Since $\phi(0) = 1$, $\phi(1) = 0$ and $\phi$ is strictly decreasing we have that $\phi$ is strictly positive on $[0,1)$ and hence the functions $w_{0}$ and $w_{1}$ are continuous and strictly increasing and thus are invertible. Define $f\colon [0,1]  \to  [0,1] $ by 
%<% Displayed Equation: Definition of the expanding factor map.
\begin{align}
  \label{eqn:Factor Map Definition}
  f(x) =
  \left\{ 
    \begin{array}[h]{ll}
      w_{0}^{-1}(x), & \text{ if } x \in [0,A);\\
      w_{1}^{-1}(x), & \text{ if } x \in [A,1].\\
    \end{array}
  \right.
\end{align}
%>%
\par
%>%
%<% Paragraph 3
Using \Cref{eqn:Factor Map Branch Left,eqn:Factor Map Branch Right,eqn:Factor Map Definition} we compute
%<% Displayed Equation: Derivative of the expanding factor map.
      \begin{align}
      \label{eqn:Factor Map Derivative}
      Df(x) 
      &=
      \left\{
        \begin{array}[h]{ll}
          \left[\phi\left( f(x) \right)\right]^{-1}, & \text{ if } x \in [0,A);\\
            \left[1-\phi\left( f(x) \right)\right]^{-1}, & \text{ if } x \in (A,1].
          \end{array}
          \right.\\
      \label{eqn:Factor Map Second Derivative}
      D^2f(x) 
      &=
      \left\{
        \begin{array}[h]{rl}
		-D\phi\left( f\left( x \right) \right)\left[ Df(x) \right]^3, &\text { if } x \in [0,A);\\
		D\phi\left( f\left( x \right) \right)\left[ Df(x) \right]^3, & \text{ if } x \in (A,1].
          \end{array}
          \right.
        \end{align}
        %>%
The alternative representation of $g_{x}$ below follows from the displayed equation above and \Cref{eqn:Fiber Map}.
\begin{align}
  \label{eqn:Fiber Map Alternate}
    g_{x}(y) &=
    \left\{
      \begin{array}[h]{ll}
        \frac{y}{Df(x)}, &\text{if } x \in [0,A);\\
        1-\frac{1-y}{Df(x)}, &\text{if } x \in (A,1].
      \end{array}
    \right.
\end{align}
Taking partial derivatives of the displayed equation above we obtain the displayed equations below. To avoid confusion we write $g(x,y)$ instead of $g_{x}(y)$ to emphasise that $g\colon[0,1]^2\to [0,1]$.
\begin{align}
  \label{eqn:Fiber Map Partial_x}
  \partial_{x} g(x,y) &=
    \left\{
      \begin{array}[h]{ll}
        -y\frac{D^2f(x)}{\left[Df(x)\right]^2}, &\text{if } x \in [0,A);\\
        (1-y)\frac{D^{2}f(x)}{\left[Df(x)\right]^2}, &\text{if } x \in (A,1].
      \end{array}
    \right.\\
  \label{eqn:Fiber Map Partial_y}
  \partial_{y} g(x,y) &= \frac{1}{Df(x)}
\end{align}
\par
%>%
%<% Paragraph 4
Note that $Df(x)$ approaches $\infty$ as $x$ approaches $A$ from the left or from the right. From \Cref{eqn:Factor Map Definition} we see that $f(0)=0$ and $f(1) =1$. From \Cref{eqn:Factor Map Derivative} we see that $Df(0) = Df(1) = 1$ and therefore $f$ has neutral fixed points at $0$ and $1$. It also follows from \Cref{eqn:Factor Map Derivative} that $Df(x)\ge 1$ for all $x\neq A$, therefore $f$ is an expanding map.\par
%>%
%<% Paragraph 5
It should be noted that for $x$ near $0$, the expanding factor $f$ is approximately $x\mapsto x(1+cx^{\alpha_{0}})$, with similar behavior near $x=1$. From \cite{Pianigiani1980} Theorem 3 we might only expect a finite invariant measure for $\alpha >1$, however $f$ does not have bounded distortion near $x =A$ so the main theorem \cite{Pianigiani1980} from does not apply. Note that $f$ is the factor, by projection onto the first coordinate, of $B$ which preserves two-dimensional Lebesgue measure. It follows that $f$ must preserve one-diminsional Lebesgue measure. In these examples unbounded distortion near $x=A$ balances slow escape from the indifferent fixed points at $x=0$ and $x=1$. The map $f$ associated to an ICF with contact exponent $\alpha$ preserves one-dimnsional Lebesgue measure for any $\alpha>0$.
%>%
%>%
%<% Subsection: The exact rate of escape from indifferent fixed points 
\subsection{Exact Rate of Escape from Indifferent Fixed Points}
\label{sec:Escape From Fixed Points}
%<% Paragraph 1
In this section we are concerned with refining asymptotic estimates from \cite{Bose&Murray2013}. We begin by setting notation. 
%>%
%<% Definition: Asymptotic notation.
\begin{dfn}
  Suppose that $f$ and $g$ are positive real valued functions. 
  \begin{itemize}
    \item We say that $f(x)\asymp g(x)$ as $x\to a$ if
      \[
        0<\liminf_{x\to a} \frac{f(x)}{g(x)} \le \limsup_{x\to a} \frac{f(x)}{g(x)} <\infty.
      \]
    \item We say that $f(x)\in O(g(x))$ as $x\to a$ if
      \[
        \limsup_{x\to a} \frac{f(x)}{g(x)} <\infty.
      \]
    \item We will say that $f(x)\sim g(x)$ as $x\to a$ if 
      \[
        \lim_{x\to a} \frac{f(x)}{g(x)} = 1.
      \]
    \item We will say that $f(x)\in o(g(x))$ as $x\to a$ if 
      \[
        \limsup_{x\to a} \frac{f(x)}{g(x)} = 0.
      \]
  \end{itemize}
  We will often abuse notation and let $O(g(x))$ (resp. $o(g(x))$) denote and arbitrary function $h$ such that $h(x) \in O(g(x))$ (resp. $h\in o(g(x))$) as $x \to a$. Note that $f(x)\asymp g(x)$ as $x \to a$ if and only if $f(x) \in O(g(x))$ and $g(x) \in O(f(x))$ as $x \to a$. Similarly $f(x) \sim g(x)$ as $x \to a$ if and only if $f(x) = g(x)\left( 1+o(1)\right)$.\par
\end{dfn}
%>%
%<% Paragraph 2
In this section we refine asymptotic estimates of the form $f(x) \asymp g(x)$ as $x\to a$ from \cite{Bose&Murray2013} to obtain asymptotic estimates of the form $f(x) \sim g(x)$ as $x \to a$.\par
%>%
%<% Paragraph 3
Throughout this section $f\colon [0,1]\to[0,1]$ will be the expanding factor map associated to an intermittent cut function with contact exponents $\alpha_{0}$ and $\alpha_{1}$, and contact constants $c_0$ and $c_{1}$.The results of this section are more precise versions of the results contained in \cite{Bose&Murray2013} Lemma 1. These refinements are need to prove limit theorems when the rate of decay of correlations is not summable.\par
%>%
%<% Paragraph 4
We begin by setting notation and collecting a few facts. The map $f$ has two smooth onto branches and $Df(x) > 1$ for $x \in (0,A) \cup (A,1)$, therefore there exist a unique period-2 orbit $\left\{ p,q \right\}$ such that $0<p<A<q<1$, \emph{i.e.}
%<% Displayed Equation: Period-2 orbit
\begin{align}
\label{eqn:Period-2 Orbit}
f(p) = q,
& \quad f(q) = p.
\end{align}
%>%
For all $n\ge 0$ define,
%<% Displayed Equation: Points associated to the period-2 orbit.
\begin{align}
\label{eqn:Period-2 Orbit Preimages Left}
p_{n} &= w_{0}^{n}(p), 
\quad
q_{n} = w_{1}^{n}(q),\\
\label{eqn:Period-2 Orbit Preimages Right}
p_{n+1}^{\circ} &= w_{1}\left(p_{n}\right),
\quad
q_{n+1}^{\circ} = w_{0}\left(q_{n}\right).
\end{align}
%>%
\par
%>%
%<% Paragraph 5
By \Cref{eqn:Factor Map Definition} $w_0$ and $w_1$ are inverses of the branches of $f$. For all $n \ge 0$,
%<% Displayed Equation: Orbit structure  of points associated with the period-2 orbit.
\begin{eqnarray}
\label{eqn:Period-2 Orbit Structure Left}
f\left(p_{n+1}\right) = p_{n},
&
f\left(q_{n+1}\right) = q_{n},\\
\label{eqn:Period-2 Orbit Structure Right}
f\left(p_{n+1}^{\circ}\right) = p_{n},&
f\left(q_{n+1}^{\circ}\right) = q_{n}.
\end{eqnarray}
%>%
\par
%>%
%<% Paragraph 6
For each $n \ge 0$, intervals are mapped onto one another by $f$ in the following pattern,
%<% Displayed Equation: Orbit structure of intervals associated with the period-2 orbit.
\begin{align}
\label{eqn:Interval Orbit Structure}
\left[p_{n+2}^{\circ} ,p_{n+1}^{\circ}\right] 
&\mapsto
\left[p_{n+1} ,p_{n}\right] 
\mapsto
\left[p_{n} ,p_{n-1}\right] 
\mapsto \dots \mapsto
\left[ p_1,p_0 \right]
\mapsto
\left[ p,q \right]\\
\left[q_{n+1}^{\circ} ,q_{n+2}^{\circ}\right] 
&\mapsto
\left[q_{n} ,q_{n+1}\right] 
\mapsto
\left[q_{n-1} ,q_{n}\right] 
\mapsto \dots \mapsto
\left[ q_0,q_1 \right]
\mapsto
\left[ p,q \right]\notag
\end{align}
%>%
\par
%>%
%<% Paragraph 7
Using \Cref{eqn:Factor Map Branch Left,eqn:Factor Map Branch Right} it is easy to check that for all $n\ge0$,
%<% Displayed Equation: Order structure of points associated to period-2 orbit.
\begin{eqnarray}
\label{eqn:Period-2 Orbit Order}
0<p_{n+1}<p_{n},
&
q_{n} < q_{n+1} < 1.
\end{eqnarray}
%>%
\par
%>%
%<% Paragraph 8
Both of the maps $w_{0}$ and $w_{1}$ are increasing. For all $n\ge 1$,
\begin{eqnarray}
\label{eqn:Period-2 Orbit Order Interior}
A<p_{n+1}^{\circ}<p^{\circ}_{n}<q,
&
p<q_{n}^{\circ} < q_{n+1}^{\circ} < A.
\end{eqnarray}
\par
%>%
%<% Paragraph 9 
The following is a refinement of Lemma 1 from \cite{Bose&Murray2013}.
%>%
%<% Lemma: Exact asymptotic behavior of points associated to the period-2 orbit.
%<% Statement
\begin{lem}
\label{Lem:Period-2 Orbit Preimage Asymptotics}
As $n \to \infty$,
%<% Displayed Equation: Exact asymptotic behavior of points associated to the period-2 orbit.
\begin{align}
  \label{eqn:Period-2 Orbit Preimage Asymptotics Exterior Left}
  p_{n} 
  &\sim
  \left( \tfrac{\alpha_{0}+1}{c_{0}\alpha_{0}} \right)^{\frac{1}{\alpha_{0}}} \left( \tfrac{1}{n} \right)^{\frac{1}{\alpha_{0}}} \\
  \label{eqn:Period-2 Orbit Preimage Asymptotics Exterior Right}
  1-q_{n} 
  &\sim
  \left( \tfrac{\alpha_{1}+1}{c_{1}\alpha_{1}} \right)^{\frac{1}{\alpha_{1}}} \left( \tfrac{1}{n} \right)^{\frac{1}{\alpha_{1}}} \\
  \label{eqn:Period-2 Orbit Increment Asymptotics Exterior Left}
  p_{n} - p_{n+1}
  &\sim
  \tfrac{1}{\alpha_{0}}
  \left(
  \tfrac{\alpha_{0}+1}{c_{0}\alpha_{0}}
  \right)^{\tfrac{1}{\alpha_{0}}}
  \left(\tfrac{1}{n}\right)^{1+\tfrac{1}{\alpha_{0}}}\\
  \label{eqn:Period-2 Orbit Increment Asymptotics Exterior Right}
  q_{n+1}- q_{n}
  &\sim
  \tfrac{1}{\alpha_{1}}
  \left(
  \tfrac{\alpha_{1}+1}{c_{1}\alpha_{1}}
  \right)^{\tfrac{1}{\alpha_{1}}}
  \left(\tfrac{1}{n}\right)^{1+\tfrac{1}{\alpha_{1}}}\\
  \label{eqn:Period-2 Orbit Preimage Asymptotics Interior Right}
  p_{n}^{\circ} - A
  &\sim
  \tfrac{1}{\alpha}
  \left( \tfrac{\alpha+1}{c\alpha} \right)^{\frac{1}{\alpha_{0}}} 
  \left( \tfrac{1}{n} \right)^{1+\frac{1}{\alpha_{0}}}\\
  \label{eqn:Period-2 Orbit Preimage Asymptotics Interior Left}
  A-q_{n}^{\circ}
  &\sim
  \tfrac{1}{\alpha}
  \left( \tfrac{\alpha+1}{c\alpha} \right)^{\frac{1}{\alpha_{1}}} 
  \left( \tfrac{1}{n} \right)^{1+\frac{1}{\alpha_{1}}}\\
  \label{eqn:Period-2 Orbit Increment Asymptotics Interior Right}
  p_{n}^{\circ} - p_{n+1}^{\circ}
  &\sim
  \tfrac{c_{0}}{\alpha_{0}}
  \left(
  \tfrac{\alpha_{0}+1}{c_{0}\alpha_{0}}
  \right)^{1+\tfrac{1}{\alpha_{0}}}
  \left(\tfrac{1}{n}\right)^{2+\tfrac{1}{\alpha_{0}}}\\
  \label{eqn:Period-2 Orbit Increment Asymptotics Interior Left}
  q_{n+1}^{\circ} - q_{n}^{\circ}
  &\sim
  \tfrac{c_{0}}{\alpha_{0}}
  \left(
  \tfrac{\alpha_{0}+1}{c_{0}\alpha_{0}}
  \right)^{1+\tfrac{1}{\alpha_{0}}}
  \left(\tfrac{1}{n}\right)^{2+\tfrac{1}{\alpha_{1}}}
\end{align}
%>%
\end{lem}
%>%
%<% Proof
\begin{proof}
  We begin by proving \Cref{eqn:Period-2 Orbit Preimage Asymptotics Exterior Left}, the proof of \Cref{eqn:Period-2 Orbit Preimage Asymptotics Exterior Right} is similar.
  From the definition of $\phi$ and $w_0$ we have, as $x \to 0$,
  \[
    x-w_{0}(x) = \int_{0}^{x} 1-\phi(t)\, dt = \frac{c_{0}}{\alpha_{0}+1}x^{\alpha_{0} +1} + o\left( x^{\alpha_{0}+1} \right).
  \]
Note that, as $y \to \infty$,
\[
  \frac{1}{\left(\frac{1}{y}\right)^{\frac{1}{\alpha_{0}}}-\left(\frac{1}{y+z}\right)^{\frac{1}{\alpha_{0}}}}
  =
  \frac{y^{\frac{1}{\alpha_{0}}}}{1-\left(1+\frac{z}{y}\right)^{-\frac{1}{\alpha_{0}}}}
  =\frac{\alpha_{0} y^{\frac{1}{\alpha_{0}} +1}}{z} + o\left(y^{\frac{1}{\alpha_{0}}}\right).
\]
The second equality above is obtained by computing the MacLaurin series of the middle expression divided by its numerator in terms of the variable $\tfrac{1}{y}$. Using the last two displayed equations, we obtain, as $y \to \infty$,
\[
  \frac{\left( \frac{1}{y} \right)^{\frac{1}{\alpha_{0}}} - w_{0}\left( \left( \frac{1}{y} \right)^{\frac{1}{\alpha_{0}}} \right)}{\left(\frac{1}{y}\right)^{\frac{1}{\alpha_{0}}}-\left(\frac{1}{y+z}\right)^{\frac{1}{\alpha_{0}}}}
  =
  \frac{\alpha_{0} c_{0}}{\alpha_{0} +1} \frac{1}{z} + o(1)
\]
Setting $z = \frac{\alpha_{0} c_{0}}{\alpha_{0} +1}$ and $\left( \frac{1}{y} \right)^{\frac{1}{\alpha_{0}}} = p_{k}$, we obtain, as $k\to \infty$
\[
  \frac{p_{k}-p_{k+1}}{p_{k} - \left( \frac{1}{y+z} \right)^{\frac{1}{\alpha_{0}}}} = 1 + o(1).
\]
An induction argument shows that, as $k \to \infty$, for all $j\ge 1$,
\[
  \frac{p_{k}-p_{k+j}}{p_{k} - \left( \frac{1}{y+jz} \right)^{\frac{1}{\alpha_{0}}}} = 1 + o(1).
\]
Rearranging yields, $p_{k+j} \sim\left(y+jz \right)^{-\frac{1}{\alpha_{0}}}$ as $k \to \infty$.
Note that $z(k+j) \sim y+jz$ as $j \to \infty$, therefore $p_{k+j} \sim z^{-\frac{1}{\alpha_{0}}}\left(k+j\right)^{-\frac{1}{\alpha_{0}}}$ as $j\to \infty$ faster than $k\to \infty$. Letting $n = k+j$ we conclude that
\[
  p_{n} 
  \sim
  \left( 
  \tfrac{\alpha_{0}+1}{\alpha_{0} c_{0}} 
  \right)^{\frac{1}{\alpha_{0}}} 
  \left( 
  \tfrac{1}{n} 
  \right)^{\frac{1}{\alpha_{0}}}. 
\]
This completes the proof of \Cref{eqn:Period-2 Orbit Preimage Asymptotics Exterior Left}.\par

\Cref{eqn:Period-2 Orbit Increment Asymptotics Exterior Left} follows from \Cref{eqn:Period-2 Orbit Preimage Asymptotics Exterior Left} since 
  \[
    \left( \tfrac{1}{n} \right)^{\frac{1}{\alpha_0}}
      -
    \left( \tfrac{1}{n+1} \right)^{\frac{1}{\alpha_0}}
    \sim 
    \tfrac{1}{\alpha_0} \left( \tfrac{1}{n} \right)^{1+\frac{1}{\alpha_0}}.
  \]

  To Prove \Cref{eqn:Period-2 Orbit Preimage Asymptotics Interior Right} we note that by \Cref{eqn:Period-2 Orbit Preimages Right,eqn:Factor Map Branch Right,eqn:Cut Function Expansion Left} we have 
\begin{align*}
  p_{n}^{\circ} - A
  &= 
  w_{1}(p_{n-1}) - w_{1}\left( 0 \right)\\
  &=
  \int_{0}^{p_{n-1}} 1-\phi(t) \, dt\\
  &=
  \int_{0}^{p_{n-1}} c_{0}t^{\alpha_{0}} + h(t) \, dt\\
  &=
  \left(
  \tfrac{c_{0}}{\alpha_{0}+1}
  \right)
  \left( p_{n-1}^{\alpha_{0}+1} \right) 
  + o\left( \int_{0}^{p_{n-1}} t^{\alpha_{0}} \, dt \right)\\
  &\sim
  \tfrac{1}{\alpha_{0}}
  \left( \tfrac{\alpha_{0}+1}{c_{0}\alpha_{0}} \right)^{\frac{1}{\alpha_{0}}} 
  \left( \tfrac{1}{n} \right)^{1+\frac{1}{\alpha_{0}}}
\end{align*}
\end{proof}
\begin{lem}
\label{lem:Point Asymptotics}
Suppose $n\ge 0$ and that $(x,y) \in [p,q]\times[0,1]$.
For all $1\le k\le n+1$, let $(x_k,y_k) = B^{k}(x,y)$. 
\begin{enumerate}[i.]
  \item If $x \in \left[p_{n+2}^{\circ},p_{n+1}^{\circ}\right]$, then as $n-k \to \infty$,
    \begin{align}
      \label{eqn:Point Asymptotics Exterior Left}
      x_{k}
      &\sim
      \left( \tfrac{\alpha_{0} +1}{c_{0}\alpha_{0}} \right)^{\frac{1}{\alpha_{0}}}
      \left(\tfrac{1}{n-k+2}  \right)^{\frac{1}{\alpha_{0}}},\\
      y_{k}
      &\sim
      \left(1-\tfrac{k+1}{n}\right)^{1+\frac{1}{\alpha_{0}}}.\notag
    \end{align}
  \item If $x \in \left[q_{n+1}^{\circ},q_{n+2}^{\circ}\right]$, then as $n-k \to \infty$,
    \begin{align}
      \label{eqn:Point Asymptotics Exterior Right}
      1-x_{k}
      &\sim
      \left( \tfrac{\alpha_{1} +1}{c_{1}\alpha_{1}} \right)^{\frac{1}{\alpha_{1}}}
      \left(\tfrac{1}{n-k+2}  \right)^{\frac{1}{\alpha_{1}}},\\
      y_{k}
      &\sim
      \left(\tfrac{k+1}{n}\right)^{1+\frac{1}{\alpha_{1}}}.\notag
    \end{align}
\end{enumerate}
\end{lem}
\begin{proof}
We will only prove the asymptotic for $x \in \left[p_{n+2}^{\circ},p_{n+1}^{\circ}\right]$ the case of $x \in \left[q_{n+1}^{\circ},q_{n+2}^{\circ}\right]$ being similar. 
Throughout this proof we will suppress subscripts ($\alpha := \alpha_{0}$ and $c:=c_{0}$).
By \Cref{eqn:Interval Orbit Structure}, $x_{k} \in \left[ p_{n-k+2},p_{n-k+1} \right]$. By \Cref{eqn:Period-2 Orbit Preimage Asymptotics Exterior Left}, as $n-k \to \infty$,
\[
  p_{n-k+2} \sim  
  \left( \tfrac{\alpha+1}{c\alpha} \right)^{\frac{1}{\alpha}} \left( \tfrac{1}{n-k+2} \right)^{\frac{1}{\alpha}}.
\]
By \Cref{eqn:Period-2 Orbit Increment Asymptotics Exterior Left}, as $n-k \to \infty$,
\[
  x_{k} - p_{n-k+2} \le p_{n-k+1} - p_{n-k+2} = o\left( \tfrac{1}{n-k} \right)^{\frac{1}{\alpha}}.
\]
This verifies the claimed asymptotic behavior of $x_{k}$.\par

Recall \Cref{eqn:Fiber Map,eqn:Fiber Map Iterates Left}, and note that for $k\ge 2$
\[
  y_{k}
  =
  \left[ \phi(x_{1}) + \left( 1-\phi\left( x_{1} \right) \right)y\right]
  \prod_{j=2}^{k} \phi\left( x_{j} \right)
\]
and $y_{1}$ can be obtained by omitting the product in the equation above. Applying \Cref{eqn:Cut Function Expansion Left} and expanding $\log(1-t)$ about $t = 0$, we see that as $t \to 0$ 
\[
  \log\left( \phi\left( t \right) \right) = \log\left( 1-ct^{\alpha} + h(t) \right) \sim -ct^{\alpha}.
\]
Applying the asymptotic for $x_{k}$ from above we obtain, as $n-k \to \infty$,
\[
  \log\left( \phi\left( x_{j} \right) \right) \sim -\left( \tfrac{\alpha+1}{\alpha} \right)\left( \tfrac{1}{n-j+2} \right).
\]
It follows that, as $n-k \to \infty$ 
\[
  \sum_{j=2}^{k}
  \log\left( \phi\left( x_{j} \right) \right) 
  \sim 
  -\left( \tfrac{\alpha+1}{\alpha} \right)
  \sum_{j=2}^{k}
  \tfrac{1}{n-j+2}
  \sim
  \tfrac{\alpha+1}{\alpha} \log\left( \tfrac{n-k+1}{n} \right).
\]
Therefore,
\[
  \prod_{j=2}^{k} \phi\left( x_{j} \right) \sim \left( 1-\tfrac{k+1}{n} \right)^{1+\tfrac{1}{\alpha}}.
\]
%Applying the expansion $(1-t)^{a} = 1-at +o(t)$ we obtain
%\[
%	\prod_{j=2}^{k} \phi\left( x_{j} \right) 
%	\sim 
%	1-\left(\tfrac{\alpha+1}{\alpha}\right)\left(\tfrac{k+1}{n}\right).
%\]
Noting that $\phi(x_{1}) = 1 + o\left( \tfrac{1}{n} \right)$ we see that, as $n-k \to \infty$,
\[
  y_{k} \sim \left( 1-\tfrac{k+1}{n} \right)^{1+\frac{1}{\alpha}},
\]
as desired.
\end{proof}
%>%
%>%
%>%
%>%
%<% Section: Induced Map
\section{Induced Map}
\label{sec:Induced Map}
%<% Paragraph: Introduction
In this section we will construct an induced map that will enjoy uniform hyperbolicity and bounded distortion.\par
%>%
%<% Paragraph: Definition of the base
Consider an Intermittent Baker's Transformation $B\colon [0,1]^{2} \to  [0,1]^{2}$ as defined in \Cref{sec:Maps}. Let $f$ denote the expanding factor that was described in \Cref{sec:Factor Map} and let $\left\{ p,q \right\}$ denote the period-2 orbit described in \Cref{eqn:Period-2 Orbit}. Define the set
\begin{align}
  \label{eqn:Base Definition}
  \Lambda &= [p,q] \times [0,1].
\end{align}
We will refer to $\Lambda$ as the \emph{base} and consider first returns to $\Lambda$.\par
%>%
%<% Paragraph: Definition of return time function
Define the \emph{return time function} $r\colon \Lambda \to \N \cup \left\{ \infty \right\}$ by
\begin{align}
  \label{eqn:Return Time Definition}
  r(x,y)  = \inf\left\{ n \in \N\cup\left\{ \infty \right\} : B^{n}(x,y) \in \Lambda \right\}.
\end{align}
\par
%>%
%<% Paragraph: Definition of induced map
The \emph{induced map} $T \colon \Lambda  \to  \Lambda $, defined by
\begin{align}
  \label{eqn:Induced Map Definition}
  T(x,y) = B^{r(x,y)}(x,y),
\end{align}
maps a point in $\Lambda$ to the first point along its $B$-orbit that lands in $\Lambda$.\par
%>%
%<% Paragraph: Skew product property of induced map
Given a point $(x,y)$ the first coordinate of a $B^{n}(x,y)$ is independent of $y$ for all $n\ge0$, similarly membership of $(x,y)$ in $\Lambda$ does not depend on $y$. We conclude that $r(x,y)$ does not depend on $y$. It follows that
\begin{align}
  \label{eqn:Induced Map Skew Product}
  T(x,y) = B^{r(x)}(x,y) = \left( f^{r(x)}(x), g^{\left( r(x) \right)}_{x}(y) \right).
\end{align}
We see that $T$ is a skew product and define a \emph{factor map} $u\colon [p,q]  \to  [p,q] $ and \emph{fibre maps} $v_{x}\colon[0,1] \to [0,1]$ for each $x \in [p,q]$ by,
\begin{align}
  \label{eqn:Induced Factor Map}
  u(x) &= f^{r(x)}(x),\\
  \label{eqn:Induced Fiber Map}
  v_{x}(y) & = g^{(r(x))}_{x}(y).
\end{align}
\par
%>%
%<% Paragraph: Definition of base measure
Let $\lambda$ denote the conditional measure on $\Lambda$, defined by
\begin{equation}
  \label{eqn:Base Measure Definition}
  \lambda(E) = \frac{\Leb(E\cap\Lambda)}{\Leb(\Lambda)}.
\end{equation}
\par
%>%
%<% Paragraph: Measure of return sets
Note that by \Cref{eqn:Interval Orbit Structure} we have, for each $n\ge0$,
\begin{equation}
  \label{eqn:Return Time Cells}
  \left[ r=n+2 \right] = \left(\left(q_{n+1}^{\circ},q_{n+2}^{\circ}\right]\cup\left[p_{n+2}^{\circ},p_{n+1}^{\circ}\right) \right)\times [0,1].
\end{equation}
It follows from \Cref{Lem:Period-2 Orbit Preimage Asymptotics} that,
\begin{equation}
  \label{eqn:Measure Asymptotics}
  \lambda\left[ r=n \right] \asymp \left( \frac{1}{n} \right)^{\frac{1}{\alpha} +2},
\end{equation}
where $\alpha = \max\left\{ \alpha_{0},\alpha_{1} \right\}$.
\par
%>%
%<% Paragraph: Induced map derivatives
%<% Lemma: Induced Factor Map Derivative
\begin{lem}
  \label{lem:Induced Factor Map Derivative}
  If $x\in (A,q)$, then
  \begin{align}
    \label{eqn:Induced Factor Map Derivative Right}
    Du(x) &= \left[ 1-\phi\left( f(x) \right) \right]^{-1} \prod_{k=2}^{r(x)} \left[\phi\left( f^{k}(x) \right)\right]^{-1}
  \end{align}
  If $x\in(p,A)$, then
  \begin{align}
    \label{eqn:Induced Factor Map Derivative Left}
    Du(x) &= \left[\phi\left( f(x) \right) \right]^{-1} \prod_{k=2}^{r(x)} \left[1-\phi\left( f^{k}(x) \right)\right]^{-1}
  \end{align}
\end{lem}
\begin{proof}
  Suppose that $n\ge0$. 
  If $(x,y) \in \Lambda$ such that $x \in \left[p^{\circ}_{n+2},p^{\circ}_{n+1}\right)$, then $x \in [A,1]$, $\left\{B^{k}(x,y):1\le k \le n+1  \right\} \subset [0,p)\times[0,1]$, and $B^{n+2}(x,y) \in \Lambda$. 
  Using the relationship between in induced factor map $u$ and the factor map $f$ from \Cref{eqn:Induced Factor Map} and the derivative formulas from \Cref{eqn:Factor Map Derivative} we apply the chain rule to verify \Cref{eqn:Induced Factor Map Derivative Right}. A similar argument verifies \Cref{eqn:Induced Factor Map Derivative Left}
\end{proof}
%>%
%<% Lemma: Induced fiber map derivatives
\begin{lem}
  \label{lem:Induced Fiber Map Identity}
  If $x \in (A,q)$, then
  \begin{align}
  \label{eqn:Induced Fiber Map Identity right}
    v(x,y) = g(x,y) \frac{Df}{Du}(x).
  \end{align}
  If $x \in (p,A)$, then
  \begin{align}
  \label{eqn:Induced Fiber Map Identity left}
  1-v(x,y) = \left[1-g(x,y)\right] \frac{Df}{Du}(x).
  \end{align}
\end{lem}
\begin{proof}
  This follows by inspecting \Cref{eqn:Factor Map Derivative,lem:Induced Factor Map Derivative,eqn:Induced Fiber Map}. Intuitively $\frac{Df}{Du}$ collects all of the contractions that are applied by the dynamics after the first affine operation on the fiber.
\end{proof}
\begin{lem}
  \label{lem:Induced Fiber Map Derivative}
  If $x \in (A,q)$, then 
  \begin{align}
    \label{eqn:Induced Fiber Map Derivative Right}
    \partial_{x}v(x,y) 
    =& 
    (1-y)
    \frac{D^2u(x)}{\left[ Du(x) \right]^2}
    \frac{Df(x)}{Du(x)}
    -
    g(x,y)
    \frac{D^2u(x)}{\left[ Du(x) \right]^2}
    Df(x)\\
    &
    +
    g(x,y)
    D\phi\left( f(x) \right)
    \frac{\left[ Df(x) \right]^3}{Du(x)}\notag
  \end{align}
  If $x \in (p,A)$, then 
  \begin{align}
    \label{eqn:Induced Fiber Map Derivative Left}
    \partial_{x}v(x,y) 
    &= 
    y
    \frac{D^2u(x)}{\left[ Du(x) \right]^2}
    \frac{Df(x)}{Du(x)}
    +
    [1-g(x,y)]
      \frac{D^2u(x)}{\left[ Du(x) \right]^2}
      Df(x)\\
    &+
    [1-g(x,y)]
    D\phi\left( f(x) \right)
    \frac{\left[ Df(x) \right]^3}{Du(x)}\notag
  \end{align}
\end{lem}
\begin{proof}
  Differentiate \Cref{eqn:Induced Fiber Map Identity right,eqn:Induced Fiber Map Identity left} respectively and apply \Cref{eqn:Fiber Map Partial_x,eqn:Factor Map Second Derivative}.
\end{proof}
%>%
%>%
%<% Paragraph: Definition of factor measure
Define the projection $\mu$ of the measure $\lambda$ onto $[p.q]$, by 
\begin{equation}
  \label{eqn:Factor Measure Definition}
  \mu(E) = \lambda\left( E\times [0,1] \right).
\end{equation}
\par
%>%
%<% Paragraph: Discussion of transfer operator
Recall the usual transfer operator $T_{*}$ acting on measures is defined for measurable $E$ by $T_{*}\nu(E) = \nu\left( T^{-1}E \right)$. The transfer operator induces the Perron-Frobenius operator $P$ on $L^{1}(\lambda)$. Given $\eta \in L^{1}\left( \lambda \right)$ and a measurable set $E$, define $\nu(E) =\int_{E} \eta \,d\lambda$, then the Perron-Frobenius operator is defined by $P\eta = \frac{dT_{*}\nu}{d\lambda}$ where the right hand side is a Radon-Nikodym derivative. We note that $T$ is invertible and preserves $\lambda$, and therefore $\frac{dT_{*}\nu}{d\lambda} = \eta \circ T^{-1}$. For this reason we will abuse notation and use $T_{*}$ to denote both the transfer operator and the Perron-Frobenius operator associated to $T$, that is, for all $\eta \in L^{1}(\lambda)$,
\begin{equation}
  \label{eqn:Transfer Operator Definition}
  T_{*}\eta = \eta \circ T^{-1}.
\end{equation}
\par
%>%
%<% Paragraph: Iterate notation for induced map
When we refer to iterates of $T$ we will use the notation $v^{(k)}_{x}$ defined analogously to \Cref{eqn:Fiber Map Iterates Left} so that we have.
\begin{equation}
  \label{eqn:T-skew}
  T^{k}(x,y) = \left( u^k(x), v^{(k)}_{x}(y) \right)
\end{equation}
\par
%>%
%<% Paragraph: Preservation of the factor measure
The map $u$ preserves $\mu$.\par
%>%
%<% Paragraph: Higher order returns
In what follows it will be convenient to define the $k$-th return time $r^{(k)}\colon \Lambda \to \N \cup \left\{\infty\right\}$ by,
\begin{align}
  r^{(1)}(x,y) &= r(x,y)\notag\\
  \label{eqn:Return Time Recursive}
  r^{(k+1)}(x,y) &= r^{(k)}(x,y) + r\left( T^{k}(x,y) \right).
\end{align}
Note that if $n = r^{(k)}(x,y)$, then $n$ is the smallest positive integer so that the set $\left\{ B^j(x,y): j=1,\dots,n \right\}$ contains $k$ points in $\Lambda$.
%>%
%<% Subsection: Dynamical partitions
\subsection{Dynamical Partitions}
\label{sec:Partitions}
Our anisotropic Banach spaces will be built with respect to stable and unstable curves for the IBT. 
Since $T$ is a skew product, it is easy to check that vertical lines form an equivariant family of stable curves for $T$. For convenience we introduce notation. For every $x\in[p,q]$, define 
\begin{equation}
  \label{eqn:Vertical Line Notation}
  \ell(x) = \left\{ x \right\}\times [0,1].
\end{equation}
With this notation equivariance takes the form
\begin{equation}
  \label{eqn:Vertical Line Equivariance}
  T\left(\ell(x) \right) \subset \ell(u(x)).
\end{equation}
It is routine to check that for every $x \in [p,q]$ the map $v_{x}\colon \ell(x) \to \ell(u(x))$ is an affine contraction by at least $\beta$.\par

The next lemma characterizes unstable curves for $T$.
\begin{lem}
  \label{lem:Unstable Curves}
  There is an equivariant family $\Gamma$ of unstable curves for $T$ such that, each curve is the graph of a function in $C^{1}\left( [p,q],[0,1] \right)$, the family is bounded in the $C^{1}$ norm, and the family forms a partition of $\Lambda$.
\end{lem}
\begin{proof}
  The proof is a standard but involved application of graph transformations.
  See \cite{Robinson1999} Chapter 12 or \cite{Chart2016} Lemma 5.4.15.
\end{proof}
We define $\gamma \colon \Lambda \to \Gamma$ by,
\begin{equation}
  \label{eqn:Unstable Curve Notation}
  \gamma(x,y) \in \Gamma \text{ such that } (x,y) \in\gamma(x,y).
\end{equation}
Since $\Gamma$ is a partition $\gamma(x,y)$ is uniquely defined.\par

Note that by \Cref{eqn:Return Time Cells} the collection $\left\{ \left[ r=n \right]:n\ge 1  \right\}$ is a partition mod $ \lambda$ of $\Lambda$, as is $\left\{ \left( p,A)\times[0,1],(A,q \right)\times[0,1] \right\}$. For all $k\ge 1$ we define, 
\begin{align}
  \label{eqn:Stable Columns Partition}
  \Omega_{1} &= \left\{ \left[ r=n \right]:n\ge 1  \right\} \vee \left\{ \left( p,A)\times[0,1],(A,q \right)\times[0,1] \right\},\\
  \Omega_{k+1} &= \Omega_{1}\vee T^{-1}\Omega_{k}.\notag
\end{align}
All of these collections are partitions mod $ \lambda$ since $T$ is measure preserving. Every cell of $\Omega_{k}$ is a column of the form $[a,b)\times[0,1)$ or $(a,b]\times[0,1]$. 
We define $\omega_{k} \colon \Lambda \to \Omega_{k}$ by,
\begin{equation}
  \label{eqn:Stable Columns Notation}
  \omega_{k}(x,y) \in \Omega_{k}\text{ such that } (x,y)\in\omega_{k}(x,y).
\end{equation}
Since $\Omega_{k}$ is a partition mod $\lambda$, we have that $\omega_{k}(x,y)$ is uniquely defined for $\lambda$-a.e. $(x,y)$.
Note that $r^{(k)}$ is measurable with respect to $\Omega_{k}$.\par

Let $\hat{\Omega}_{k}$ denote the projection of $\Omega_{k}$ on to the interval $[p,q]$ by the map $(x,y) \mapsto x$. \par

Lastly we define measurable partitions $\Theta_{n}$ and maps $\theta_{n}\colon \Lambda \to \Theta_{n}$ by
\begin{align}
  \label{eqn:Unstable Strip Partition}
  \Theta_{n} &= T^{n} \Omega_{n}\\
  \label{eqn:Unstable Strip Notation}
  \theta_{n}(x,y) &\in \Theta_{n}\text{ such that } (x,y)\in\theta_{n}(x,y).
\end{align}
The cells of $\Theta_{n}$ are strips that are bounded above and below by curves in $\Gamma$ and extend across the full width of $\Lambda$.
%>%
%<% Subsection: Derivative Bounds
\subsection{Derivative Bounds}
While an IBT is non-uniformly hyperbolic, the induced map introduced in the last section enjoys uniform hyperbolicity. 
For our purposes it suffices to show that the factor map $u$ of the induced map $T$ is a well behaved interval map meaning that it enjoys uniform expansion and bounded distortion. The following lemmas from \cite{Bose&Murray2013} provide the necessary bounds.
\begin{lem}[Lemma 2 from \cite{Bose&Murray2013}]
  \label{lem:Uniform Expansion}
  If
  \begin{equation}
    \label{eqn:beta Definition}
    \beta = \sup_{t \in [p,q]} \max\left\{ \phi(t),1-\phi(t) \right\},
  \end{equation}
  then
  \begin{equation}
    \label{eqn:Expansion Bound}
    \Ninfty{\left[  Du \right]^{-1}} \le \beta.
  \end{equation}
\end{lem}
\begin{proof}
  This follows immediately from \Cref{eqn:Induced Factor Map Derivative Left,eqn:Induced Factor Map Derivative Right}. Note that every term in the product can be bounded above by 1 since $\phi$ takes values in $[0,1]$ and $f^{r(x)}(x) \in [p,q]$.
\end{proof}
\begin{lem}[Lemma 3 from \cite{Bose&Murray2013}]
  \label{lem:Distortion Bound}
  There exists $\kappa<\infty$ such that for all $k \ge 1$, if $w$ and $x$ lie in the same cell of $\hat{\Omega}_{k}$, then 
  \begin{equation}
    \label{eqn:Distortion Bound}
    \abs{
      \frac{Du^k(x)}{Du^{k}(w)} - 1
    }
    \le \kappa \abs{u^k(x) - u^k(w)}.
  \end{equation}
\end{lem}
\begin{proof}
  See \cite{Bose&Murray2013}.
\end{proof}
\begin{lem}
  \label{lem:Skew Bound}
  There exists a constant $\tau>0$ such that
  \begin{align}
  \label{eqn:Skew Bound}
    \Ninfty{\frac{\partial_{x} v^{(k)}}{Du^{k}}} \le \tau.
  \end{align}
\end{lem}
\begin{proof}
  Suppose that $k = 1$. By \Cref{eqn:Induced Fiber Map Derivative Right} we have the identity below for almost every $x \in (A,q)$.
  \begin{align*}
    \frac{\partial_{x} v(x,y)}{Du(x)} 
    &= 
    (1-y)
    \frac{D^2u(x)}{\left[ Du(x) \right]^2}
    \frac{Df(x)}{[Du(x)]^2}
    -
    g(x,y)
    \frac{D^2u(x)}{\left[ Du(x) \right]^2}
    \frac{Df(x)}{Du(x)}\\
    &
    +
    g(x,y)
    D\phi\left( f(x) \right)
    \frac{\left[ Df(x) \right]^3}{[Du(x)]^2}
  \end{align*}
  Similarly, by \Cref{eqn:Induced Fiber Map Derivative Left} we have the identity below for almost every $x\in (p,A)$.
  \begin{align*}
    \frac{\partial_{x}v(x,y)}{Du(x)} 
    &= 
    y
    \frac{D^2u(x)}{\left[ Du(x) \right]^2}
    \frac{Df(x)}{\left[Du(x)\right]^2}
    +
    [1-g(x,y)]
      \frac{D^2u(x)}{\left[ Du(x) \right]^2}
      \frac{Df(x)}{Du(x)}\\
    &+
    [1-g(x,y)]
    D\phi\left( f(x) \right)
    \frac{\left[ Df(x) \right]^3}{\left[Du(x)\right]^2}\notag
  \end{align*}
  Taking norms we obtain the bound below.
  \begin{align*}
    \Ninfty{
      \frac{\partial_{x} v}{Du} 
    }
    &\le
    \Ninfty{
      \frac{D^2u}{\left[ Du \right]^2}
    }
    \Ninfty{
    \frac{Df}{[Du]^2}
    }
    +
    \Ninfty{
     \frac{D^2u}{\left[ Du \right]^2}
    }
    \Ninfty{
      \frac{Df}{Du}
    }
    +
    \Ninfty{
	    \frac{\left[ D\phi\circ f \right]\left[ Df \right]^3}{[Du]^2}
    }
  \end{align*}
  By \Cref{lem:Distortion Bound} we have,
    \[
      \Ninfty{
        \frac{D^2u}{\left[ Du \right]^2}
      }
      \le \kappa.
    \]
  \par

  Suppose that $x \in \left(p^{\circ}_{n+2},p^{\circ}_{n+1}\right)$. Since $\phi$ is decreasing and $f(x)$ is increasing, $Df(x) = [1-\phi(f(x))]^{-1}$ is decreasing. Recall from \Cref{eqn:Interval Orbit Structure} that $f$ maps $\left(p^{\circ}_{n+2},p^{\circ}_{n+1}\right)$ onto $\left( p_{n+1},p_{n} \right)$. By the Mean Value Theorem there exists $\theta_{n} \in \left(p^{\circ}_{n+2},p^{\circ}_{n+1}\right)$ such that 
  \[
    Df(\theta_{n}) 
    = 
    \frac{p_{n} - p_{n+1}}{p^{\circ}_{n+1} - p^{\circ}_{n+2}}.
  \]
  Since $Df$ is decreasing for all $x \in \left(p^{\circ}_{n+2},p^{\circ}_{n+1}\right)$,
  \[
    Df(\theta_{n-1}) \le Df(x) \le Df(\theta_{n+1}).
  \]
  By \Cref{eqn:Period-2 Orbit Increment Asymptotics Exterior Left,eqn:Period-2 Orbit Increment Asymptotics Interior Right} we have 
  \[ 
    \frac{p_{n} - p_{n+1}}{p^{\circ}_{n+1} - p^{\circ}_{n+2}} \approx n
  \]
  We deduce that for all $x \in \left(p^{\circ}_{n+2},p^{\circ}_{n+1}\right)$,
  \[
    Df(x) \approx n.
  \]
  \par
  A similar argument shows that for all $x \in \left(p^{\circ}_{n+2},p^{\circ}_{n+1}\right)$,
  \[
    Du(x) \approx \frac{q-p}{p^{\circ}_{n+1} - p^{\circ}_{n+2}} \approx n^{2+\frac{1}{\alpha_{0}}}
  \]
  \par
  By \Cref{eqn:Cut Function Expansion Left}, $D\phi(t) = -c_0\alpha_0 t^{\alpha_{0} - 1} + o\left( t^{\alpha_{0}-1} \right)$ for $t$ near $0$. If $x \in \left(p^{\circ}_{n+2},p^{\circ}_{n+1}\right)$, then $f(x) \in \left( p_{n+1},p_{n} \right)$ by \Cref{eqn:Interval Orbit Structure}. From \Cref{eqn:Point Asymptotics Exterior Left} we see that $f(x)\approx \left( \tfrac{1}{n} \right)^{1/\alpha_{0}}$. Thus, for $x \in \left( p^{\circ}_{n+2}, p^{\circ}_{n+1} \right)$,
  \[
  D\phi\left( f\left( x \right) \right) \approx n^{1/\alpha_{0}-1}. 
  \]
  \par

  Similar arguments apply to $x \in \left( q^{\circ}_{n+1},q^{\circ}_{n+2} \right)$.
  \par

  Combining the last three displayed equations and their analogues for $x > A$ we see that $\Ninfty{\frac{Df}{Du}}$, $\Ninfty{\frac{Df}{[Du]^2}}$, and $\Ninfty{\frac{[D\phi\circ f][Df]^3}{[Du]^2}}$ are all finite. 
  We conclude that for $k=1$ \Cref{eqn:Skew Bound} holds with some constant $\tau_{0}>0$.\\

  We claim that \Cref{eqn:Skew Bound} holds with $\tau = \tau_{0}(1-\beta^2)^{-1}$. To verify this we first prove that 
  \[
    \Ninfty{\frac{\partial_{x}v^{k}}{Du^k}} \le \tau_{0} \sum_{m=0}^{k-1} \beta^{2m}.
  \]
  The claim follows by replacing the finite geometric series with the infinite geometric series.
  \par

  Suppose that the displayed inequality above holds for some $k\ge 1$. Let $DT^{k}$ denote the Jacobian of $T^{k}$. Note that 
  \[
    DT^{k} = 
    \begin{bmatrix}
      Du^k & 0\\
      \partial_{x}v & \left[ Du^k \right]^{-1}
    \end{bmatrix}.
  \]
  Since $DT^{k+1} = \left(DT^k\circ T\right)DT$ we have 
  \begin{align*}
    \partial_{x}v^{(k+1)} &=\left(\partial_{x}v^{(k)}\circ T\right)Du  + \frac{\partial_{x}v}{Du^k\circ u}\,,\\
    Du^{k+1} &= \left(Du^k\circ u\right) Du.
  \end{align*}
  Thus,
    \[
      \frac{\partial_{x}v^{(k+1)}}{Du^{k+1}}
      =
      \frac{\partial_{x}v^{(k)}}{Du^{k}}\circ T
      +
      \frac{1}{\left[Du^{k}\circ u\right]^2}
      \frac{\partial_{x}v}{Du}.
    \]
    Taking the norm of both sides of the identity above, applying the induction hypothesis, and applying \Cref{eqn:Expansion Bound} we obtain the bound below.
    \begin{align*}
      \Ninfty{
        \frac{\partial_{x}v^{(k+1)}}{Du^{k+1}}
      }
      &\le
      \Ninfty{
        \frac{\partial_{x}v^{(k)}}{Du^{k}}
      }
      +
      \Ninfty{
        \frac{1}{Du^{k}}
      }^2
      \Ninfty{
        \frac{\partial_{x}v}{Du}
      }\\
      &\le
      \tau_{0}
      \sum_{m=0}^{k-1} 
        \beta^{2m}
      +
      \tau_{0}\beta^{2k}
      =
      \tau_{0}\sum_{m=0}^{k}
        \beta^{2m}\\
      &\le \frac{\tau_{0}}{1-\beta^{2}}.
    \end{align*}

\end{proof}
%>%
%>%
%<% Section: Adapted Banach Spaces
\section{Adapted Banach Spaces}
\label{sec:Adapted Banach Spaces}
%<% Paragraph 1
In this section we will define Banach spaces $\Bw$ and $\Bs$ with anisotropic norms that are adapted to the dynamics of the induced map $T$. These spaces were first introduced in \cite{Liverani&Terhesiu2015} and are a simplified version of the norms defined in \cite{Demers&Liverani2008}. 
\par
%>%
%<% Paragraph 2
We begin by constructing a space $\ULip$ of bounded measurable functions that exhibit regularity along  unstable curves, which is one of the key properties that we will need in the space $\Bs$. This regularity is necessary in the proof of a Lasota-Yorke inequality (see \Cref{prop:Uniform Lasota-Yorke Inequality}).
\par
%>%
%<% Definition: Ulip Space
\begin{dfn}
  \label{def:Ulip Space}
  Given a bounded measurable function $\eta\colon \Lambda \to \R$, define
  \begin{align*}
    \MUL{\eta}
    &=
    \sup_{\gamma\in\Gamma} \sup_{(x,y)\neq (w,z) \in \gamma} \frac{\eta(x,y)-\eta(w,z)}{\abs{x-w}},\\
    \NUL{\eta} 
    &= \Nsup{\eta}+ \MUL{\eta},\\
    \ULip &= \left\{ \eta : \NUL{\eta}<\infty \right\}.
  \end{align*}
  Recall that $\Gamma$ is the partition of $\Lambda$ by unstable curves. 
\end{dfn}
%>%
%<% Paragraph 3
While elements of the space $\Bs$ must exhibit regularity along unstable curves to satisfy a Lasota-Yorke inequality, they will also have a distributional quality along stable lines to facilitate the proof of a compact embedding. By restricting to a stable line we will view elements of $\Bs$ and $\Bw$ as functionals on spaces of H\"older functions.
\par
%>%
%<% Definition: Lipschitz and Holder spaces
 \begin{dfn}
   Let $\Tl$ denote the space of real valued Lipschitz functions with domain $[0,1]$. 
   Fix $a \in (0,1)$ and let  $\Ta$ denote the space of real values  $a$-H\"older functions with domain $[0,1]$. Let $\NTl{\cdot}$ and $\NTa{\cdot}$ denote the norms of $\Tl$ and $\Ta$ respectively.
 \end{dfn}
%>%
 %<% Paragraph 4
 Next we define norms for the spaces $\Bs$ and $\Bw$, which should be viewed as being related to the strong operator norm on the dual spaces $\Ta^*$ and $\Tl^*$. Specifically, a bounded measurable function $\eta \colon \Lambda \to \R$ induces a functional for each vertical line $\ell(x)$ through integration. Given $\psi$ in $\Ta$ or $\Tl$ 
 \[
   \psi \mapsto \int_{0}^{1}\eta(x,y)\, \psi(y)\, dy
 \]
 defines a bounded linear functional. This motivates the following definition.
 \par
 %>%
 %<% Definition: Norms
 \begin{dfn}
   For all bounded measurable functions $\eta \colon \Lambda \to \R$ define 
   \begin{align}
     \label{eqn:Weak Norm Defintion}
     \NBw{\eta}
     &=
     \sup\left\{\int_{0}^{1} \eta(x,y) \, \psi(y) \,dy : x\in [p,q], \NTl{\psi} \le 1  \right\},\\
     \label{eqn:Stable Norm Defintion}
     \Ns{\eta}
     &=
     \sup\left\{\int_{0}^{1} \eta(x,y) \, \psi(y) \,dy : x\in [p,q], \NTa{\psi} \le 1  \right\},\\
     \label{eqn:Unstable Modulus Defintion}
     \MBs{\eta}
     &=
     \sup\left\{\int_{0}^{1} \frac{\eta(w,y) -\eta(x,y)}{\abs{w - x}}  \psi(y)\,dy : w\neq x\in [p,q], \NTl{\psi} \le 1  \right\},\\
     \label{eqn:Strong Norm Defintion}
     \NBs{\eta} 
     &=
     \Ns{\eta} + \MBs{\eta}.
   \end{align}
 \end{dfn}
 %>%
%<% Paragraph 5
 Since $\Tl \subset \Ta$ we have $\NBw{\cdot} \le \Ns{\cdot} \le \NBs{\cdot}$. 
Both $\NBs{\cdot}$ and $\NBw{\cdot}$ are bounded semi-norms on the space of Lipschitz functions. By taking quotients, $\NBs{\cdot}$ and $\NBw{\cdot}$ induce norms on quotient spaces of $\ULip$. Completing these quotient spaces with respect to their norms produces Banach spaces $\Bs$ and $\Bw$.
\par
%>%
%>%
%<% Section: Main Results
\section{Main Results}
\label{sec:Main Results}
%<% Paragraph 1
In this section we apply operator renewal theory as described in \cite{Gouezel2004,Gouezel2004-Intermittent} to obtain the rate of decay of correlation (\Cref{thm:Main Correlation}) and limit theorems (\Cref{thm:Main Limit Theorem}) for an IBT $B$.
\par
%>%
%<% Paragraph 2
For each $n\ge1$ and $k\ge1$ we define operators by
\begin{align}
\label{eqn:Exact Return Operators}
R_{n}^{(k)}\eta &= T_{*}^{k}\left( \indf{ \left\{ r^{(k)} = n \right\}} \eta \right),\\
\label{eqn:Return Operators}
B_{n}\eta &= \indf{\Lambda}B_{*}^{n}\left( \indf{\Lambda} \eta \right).
\end{align}
We will always abbreviate $R^{(1)}_{n}$ as $R_{n}$. 
The operators $R_{n}$ are a decomposition of $T_{*}$ by first return time. The operators $B_{n}$ can be viewed as a restriction of $B^{n}_{*}$ to an action on functions supported on $\Lambda$.
\par
%>%
%<% Paragraph 3
A key technical observation in operator renewal theory is that the generating functions defined by \Cref{eqn:Return Generating Function,eqn:Exact Return Generating Function} are well defined and related by \Cref{eqn:Renewal Identity}.
\begin{align}
\label{eqn:Return Generating Function}
B(z) &= I + \sum_{n=1}^{\infty} z^{n} B_{n}\\
\label{eqn:Exact Return Generating Function}
R(z) &= \sum_{n=1}^{\infty}z^n R_{n}\\
\label{eqn:Renewal Identity}
B(z) &= \left[I - R(z)\right]^{-1}
\end{align}
\par
%>%
%<% Paragraph 4
In what follows we will make use of the following identities, which are routine to check,
\begin{align}
\label{eqn:Trasfer Operator at One}
R(1) &= T_{*}\\
\label{eqn:Exact kth Return Generating Function}
\left[R(z)\right]^{k} &= \sum_{n = 1} ^{\infty} R_{n}^{(k)}z^n.
\end{align}
\par
%>%
%<% Subsection Decay of Correlations
\subsection{Decay of Correlations}
\label{sec:Rates}
%<% Paragraph 1
Heuristically, if $\eta$ is supported on $\Lambda$ and $\int_{\Lambda} \eta \neq 0$, then the push forward distributions $B^{n}_{*}\eta$ must equilibrate to a multiple $\indf{\left[ 0,1 \right]^{2}}$, which is the density for the preserved measure.
The transfer operator $B_{*}$ sends all of the mass represented by $\eta$ outside of $\Lambda$. In order for $B_{*}^{n}\eta$ to attain its limiting value of $\int_{\left[ 0,1 \right]^{2}} \eta \,d\Leb$ inside of $\Lambda$, mass must return to $\Lambda$. The amount of mass that has failed to return after $n$ steps of the dynamics is $\Leb\left[ r>n \right]$, which provides a rough estimate for how quickly the convergence $B_{*}^{n}\eta \to \indf{\left[ 0,1 \right]^{2}} \int_{\left[ 0,1 \right]^{2}} \eta\, d\Leb$ can occur. \Cref{thm:Main Correlation} shows that this rough estimate is actually sharp.\par
%>%
%<% Paragraph 2
In this section we will prove \Cref{thm:Main Correlation} by applying \cite{Gouezel2004} Theorem 1.1, which we reproduced below for the convenience of the reader. We have modified notation slightly to match the current setting.\par
%>%
%<% Theorem: Gouezel Correlation
\begin{thm}[Theorem 1.1 from \cite{Gouezel2004}]
  \label{thm:Gouezel}
  Let $B_n$ be bounded operators on $\Bs$ such that $B(z) = I+\sum_{n\ge 1} z^{n} B_{n}$ converges in $Hom(\Bs,\Bs)$\footnote{With the strong operator topology} for every $z \in \C$ with $\abs{z}<1$.
  Assume that:
  \begin{enumerate}
    \item {\bf Renewal equation:} for every $z\in \C$ with $\abs{z}<1$, $B(z) = \left( I-R(z) \right)^{-1}$ where $R(z) = \sum_{n\ge 1} z^{n} R_n$, $R_n \in Hom(\mathcal{L},\mathcal{L})$ and $\sum \|R_n\|< +\infty$.
    \item {\bf Spectral Gap:} $1$ is a simple isolated eigenvalue of $R(1)$.
    \item {\bf Aperiodicity:} for every $z\neq 1$ with $\abs{z}\le 1$, $I-R(z)$ is invertible.
  \end{enumerate}
  Let $P$ be the eigenprojection of $R(1)$ at $1$. If $\sum_{k>n} \|R_{k}\| = O\left( 1/n^{\beta} \right)$ for some $\beta>1$ and $PR'(1)P \neq 0$, then for all $n$
  \[
          B_n = \frac{1}{\mu} P + \frac{1}{\mu^2} \sum_{k=n+1}^{+\infty} P_k +E_n
  \]
  where $\mu$ is given\footnote{Here $R'(1)$ denotes the operator $\frac{d}{dz}R|_{z=1}$.} by $PR'(1)P = \mu P$, $P_n = \sum_{l>n} PR_lP$ and $E_n \in Hom(\mathcal{L},\mathcal{L})$ satisfy
  \[
    \|E_n\|
    =
    \left\{
      \begin{array}[h]{ll}
        O\left( 1/n^{\beta} \right), & \textnormal{ if } \beta>2;\\
        O\left( \log(n)/n^{2} \right), & \textnormal{ if } \beta=2;\\
        O\left( 1/n^{2\beta-2} \right), & \textnormal{ if } 2>\beta>1.\\
      \end{array}
    \right.
  \]
\end{thm}
\par
%>%
%<% Paragraph 3
The following two propositions verify that the hypotheses of \Cref{thm:Gouezel} are satisfied and will be proved later.
\par
%>%
%<% Proposition: Convergence and Renewal Equation
\begin{prop}[Convergence and Renewal Equation]
  \label{prop:Convergence and Renewal Equation}\leavevmode
  \begin{itemize}
    \item For all $n \ge 1$, the operators $B_{n}$ and $R_{n}$ are bounded on $\Bs$. 
    \item For all $z$ in the open unit disk of $\C$, the operators $B(z)$ and $R(z)$ converge in $Hom\left( \Bs,\Bs \right)$ and satisfy $B(z) = \left( I-R(z) \right)^{-1}$.
    \item The operator $R(z)$ converges in $Hom\left( \Bs, \Bs \right)$ for $z$ in the closed unit disk of $\C$, that is $\sum_{n\ge1} \NBs{R_{n}} < \infty$.
    \item Let $\alpha = \max\left\{\alpha_{0},\alpha_{1}\right\}$. As $n \to \infty$, $\sum_{k>n} \NBs{R_{k}} = O\left(n^{-\left(1+\frac{1}{\alpha}\right)}\right)$. 
  \end{itemize}
\end{prop}
\begin{proof}
  See \Cref{sec:Convergence and Renewal Equation}.
\end{proof}
\par
%>%
%<% Proposition: Spectrial Gap and Aperiodicity
\begin{prop}[Spectrial Gap and Aperiodicity]
  \label{prop:Spectrial Gap and Aperiodicity}\leavevmode
  \begin{itemize}
    \item $1$ is a simple isolated eigenvalue of $R(1)$. 
    \item For $z\neq 1$ with $\abs{z}\le1$, $I-R(z)$ is invertible. 
    \item For $\eta \in \ULip$, the spectral projector $P$ can be computed by the formula $P\eta = \indf{\Lambda} \int_{\Lambda} \eta\,d \lambda$.
  \end{itemize}
\end{prop}
\begin{proof}
  See \Cref{Spectral Gap and Aperiodicity}
\end{proof}
\par
%>%
%<% Proof: Main Correlation Theorem
\begin{proof}[Proof of \Cref{thm:Main Correlation}]
  Suppose that $\eta$ and $\psi$ are Lipschitz functions on $\Lambda$. Let $\alpha = \max\left\{ \alpha_{0},\alpha_{1} \right\}$.\par

  We start by identifying the parameter $\beta$ from \Cref{thm:Gouezel}. By \Cref{prop:Convergence and Renewal Equation} we have 
  \[
    \sum_{k>n} \NBs{R_{k}} = O\left(\left(\frac{1}{n}\right)^{1+\frac{1}{\alpha}}\right).
  \]
  
  Therefore, $\beta = 1+\tfrac{1}{\alpha}$.\par

Next we identify the parameter $\mu$ from \Cref{thm:Gouezel}. Note that since $\eta$ is Lipschitz on $\Lambda$ we have $\eta \in \ULip$. Applying the spectral projector formula from \Cref{prop:Spectrial Gap and Aperiodicity} we obtain,
\begin{align*}
  P\frac{dR}{dz}(1)P \eta 
  &=
  P\sum_{n=1}^{\infty} nR_{n}P\eta
  =
  P\sum_{n=1}^{\infty} n\, R_{n}\indf{\lambda}\int \eta \, d\lambda\\
  &=
  P\sum_{n=1}^{\infty} nT_{*}\indf{[r=n]}\int \eta \, d\lambda 
  =
  \sum_{n=1}^{\infty} nPT_{*}\indf{[r=n]} \int \eta \, d\lambda \\
  &=
  \sum_{n=1}^{\infty} n\lambda[r=n] \indf{\Lambda}\int \eta \, d\lambda 
  =
  \sum_{n=1}^{\infty} n\lambda[r=n] P\eta
\end{align*}
By Kac's Lemma $\sum_{n=1}^{\infty} n \lambda[r=n] = \frac{1}{\Leb(\lambda)}$. Therefore $\mu = \frac{1}{\Leb(\Lambda)}$.\par

Next we identify the operators $P_{k}$ from \Cref{thm:Gouezel}. By a calculation similar to the one above,
\[
  P_{k}\eta = \sum_{l>k}PR_{l}P \eta = P\eta \sum_{l>k} \lambda\left[ r=l \right]= \lambda\left[ r>k \right]P\eta.
\]
Therefore,
  \[
    P_{k} = \lambda[r>k]P. 
  \]

From \Cref{thm:Gouezel} we obtain the expansion
\[
  B_{n}  = \Leb(\Lambda) P + \Leb(\Lambda)^2 \sum_{k>n} P_{k} + E_{n}
\]
where
\[
  \left\|E_{n}\right\| = 
  \left\{
    \begin{array}[h]{ll}
      O\left( \left(\frac{1}{n}\right)^{1+\frac{1}{\alpha}} \right), & \textnormal{ if } \alpha>1;\\
      O\left( \frac{\log(n)}{n^{2}} \right), & \textnormal{ if } \alpha=1;\\
      O\left( \left(\frac{1}{n}\right)^{2/\alpha} \right), & \textnormal{ if } \alpha<1.\\
    \end{array}
    \right.
  \]
  Since $\lambda$ is the conditional measure obtained by restricting $\Leb$ to $\Lambda$, see \Cref{eqn:Base Measure Definition}, we have for any $\eta \in L^{1}\left( \Lambda, \lambda \right)$, $\Leb(\Lambda)\int_{\Lambda} \eta \, d\lambda = \int_{\Lambda} \eta \, d\Leb$. Applying the expansion of $B_n$ that we have just obtained we see that 
  \begin{align*}
    B_{n}\eta 
    &= 
    \Leb(\Lambda)P\eta + \Leb(\Lambda)^{2}\sum_{k>n}P_{k}\eta + E_{n}\eta\\
    &= 
    \indf{\Lambda}\Leb(\Lambda)\int_{\Lambda}\eta\,d\lambda + \sum_{k>n}\indf{[r>k]}\Leb(\Lambda)^{2}\lambda[r>k] \int\eta\,d\lambda + E_{n}\eta\\
    &=
    \indf{\Lambda}\int_{\Lambda} \eta \,d\Leb + \indf{\Lambda}\sum_{k>n} \Leb\left[ r>k \right]\int_{\Lambda} \eta\,d\Leb + E_n\eta.
  \end{align*}
  Since $\eta$ and $\psi$ are Lipschitz on the square, $\indf{\Lambda}\eta \in \ULip$ and we obtain
  \[
    \int B_{n}\eta \, \psi\, d\Leb 
    = 
    \int \indf{\Lambda} B_{*}^{n} \left( \indf{\Lambda}\eta \right)\, \psi \, d \Leb 
    =
    \int \indf{\Lambda} \eta \, \left( \indf{\Lambda}\psi \right)\circ B^{n}\, d\Leb
  \]
  If $\eta$ and $\psi$ are the restrictions to $\Lambda$ of Lipschitz functions on the square, then
  \begin{align*}
    \int_{\Lambda} \eta \, \psi \circ B^{n} \,d\Leb 
    &=
    \int_{\Lambda} \eta \,d\Leb\int_{\Lambda} \psi\,d\Leb 
    + 
    \sum_{k>n} \Leb\left[ r>k \right]
    \int_{\Lambda} \eta \,d\Leb\int_{\Lambda} \psi\,d\Leb\\
    &+\int_{\Lambda} E_{n}\eta \, \psi \, d\Leb.
  \end{align*}
  Note that $\sum_{k>n} \Leb\left[ r>k \right]\asymp \left( \frac{1}{n} \right)^{\frac{1}{\alpha}}$ and that regardless of the value of $\alpha$ this decays slower than $\|E_{n}\|$.
  If $\int \eta \neq 0$ and $\int \psi \neq 0$, then 
  \begin{align}
    \label{eqn:Mean Zero Identity}
    \int_{\Lambda} \eta \, \psi \circ B^{n} \,d\Leb 
    -
    \int_{\Lambda} \eta \,d\Leb\int_{\Lambda} \psi\,d\Leb 
    &= 
    \sum_{k>n} \Leb\left[ r>k \right]
    \int_{\Lambda} \eta \,d\Leb\int_{\Lambda} \psi\,d\Leb\\
    &+\int_{\Lambda} E_{n}\eta \, \psi \, d\Leb\notag\\
    &\asymp \left( \frac{1}{n} \right)^{\frac{1}{\alpha}}\notag.
  \end{align}
  For functions with integral zero the rate of decay may be faster than $\left( \frac{1}{n} \right)^{\frac{1}{\alpha}}$.
\end{proof}
%>%
%<% Corollary: Mean Zero Summability of Correlations
\begin{cor}
  \label{cor:Mean Zero Summablity of Correlations}
  If the hypotheses of \Cref{thm:Main Correlation} are satisfied and additionally either $\int_{\Lambda}\psi = 0$ or $\int_{\Lambda} \eta = 0 $, then $Cor\left( k,\psi,\eta,B \right)$ is a summable sequence. 
\end{cor}
%>%
%>%
%<% Subsection: Limit Theorems
\subsection{Limit Theorems}
\label{sec:Limit Theorems}
%<% Paragraph 1
In this section we will select an observable $X\colon [0,1]^{2} \to \R$ with mean zero and deduce distributional limit behavior of the form
\begin{align}
  \label{eqn:General Limit Theorem}
  \frac{1}{A_{n}}\sum_{k=0}^{n-1} X\circ B^{k} 
  &\xrightarrow{dist} 
  Z, \quad \text{as } n\to \infty,
\end{align}
 where $A_{n}$ is a sequence of real numbers, and $Z$ is a real valued random variable and $B$ is an IBT with contact coefficients $c_{j}$ and contact exponents $\alpha_{j}$.\par
%>%
%<% Paragraph 2
The random variables that can arise as the limits in \Cref{eqn:General Limit Theorem} are stable distributions. Stable distributions with mean zero can be parameterized as follows. Let $p \in (1,2]$, $a >0$ and $b \in [-1,1]$. Let $St(p,a,b)$ be the distribution such that if $Z\sim St(p,a,b)$, then
\[
E\left[ e^{itZ} \right] = e^{-a\abs{t}^{p}\left( 1-b\sign(t)\tan\left( \frac{p\pi}{2} \right) \right)}.
\]
Note that if $p = 2$, then $Z$ is normally distributed with mean zero and standard deviation $\sigma = \sqrt{2a}$.\par
%>%
 %<% Paragraph 3
Below we collect a precise technical version of \Cref{thm:Main Limit Laws}. In order to state the theorem we need to define several constants.
\begin{align*}
  M_{0} &=\int_{0}^{1} X(0,y^{1+\frac{1}{\alpha_{0}}})\,dy,\\
  M_{1} &=\int_{0}^{1} X(1,y^{1+\frac{1}{\alpha_{1}}})\,dy,\\
  C_{0} &=\tfrac{\abs{M_{0}}}{\alpha_{0}\Leb(\Lambda)}
  \left( \tfrac{\abs{M_{0}}\left( \alpha_{0}+1 \right)}{c_{0}\alpha_{0}} \right)^{\frac{1}{\alpha_{0}}},\\
  C_{1} &=\tfrac{\abs{M_{1}}}{\alpha_{1}\Leb(\Lambda)}
  \left( \tfrac{\abs{M_{1}}\left( \alpha_{1}+1 \right)}{c_{1}\alpha_{1}} \right)^{\frac{1}{\alpha_{1}}}.
\end{align*}
\par
%>%
%<% Theorem: Main Limit Theorems
\begin{thm}
  \label{thm:Main Limit Theorem}
  Suppose that $X\colon \left[ 0,1 \right]^{2} \to \R$ is $\gamma$-H\"older for some $\gamma\in(0,1]$ and $\int_{[0,1]^{2}} X \,d \Leb = 0$. 
  \begin{enumerate}[i.]
    \item Suppose that\footnote{In particular if the hypotheses of \Cref{lem:Base Observable is L2} are satisfied} $\xi \in L^{2}$ and that $\xi$ is not a coboundary. Then $\sigma^{2} = \int_{\Lambda} \abs{\xi}^{2} \, d \lambda + 2 \sum_{k=1}^{\infty} \int_{\Lambda} \xi\circ T^{k} \, \xi \, d \lambda$ converges, $\sigma^{2}>0$, and as $n\to \infty$,
      \begin{align*}
        \frac{1}{\sqrt{n}}\sum_{k=0}^{n-1} X \circ B^{k} 
        &\xrightarrow{dist} 
        N(0,\sigma^{2}).
      \end{align*}
    \item Suppose that $\alpha_{0} > \alpha_{1}$, $\alpha_{0}>1$, and $M_{0} >0$. 
      Let $p = 1+\frac{1}{\alpha_{0}}$, $a = C_{0}\Gamma(1-p)\cos\left(\frac{p\pi}{2}\right)$, and $b = 1$. As $n \to \infty$,
      \begin{align*}
        \frac{1}{n^{\frac{\alpha_{0}}{\alpha_{0}+1}}}\sum_{k=0}^{n-1} X \circ B^{k} 
        &\xrightarrow{dist} 
        St(p,a,b).
      \end{align*}
    \item Suppose that $\alpha_{0} = \alpha_{1}=:\alpha$, $\alpha>1$, $M_{0}>0$ and $M_{1}<0$. 
      Let $p = 1+\frac{1}{\alpha}$, $a = \left(C_{0} + C_{1}\right)\Gamma(1-p)\cos\left(\frac{p\pi}{2}\right)$, and $b = \frac{C_{0}-C_{1}}{C_{0}+C_{2}}$. As $n \to \infty$,
      \begin{align*}
        \frac{1}{n^{\frac{\alpha}{\alpha+1}}}\sum_{k=0}^{n-1} X \circ B^{k} 
        &\xrightarrow{dist} St(p,a,b).
      \end{align*}
    \item Suppose that $\alpha_{0} = \alpha_{1} = 1$, $M_{0} \neq 0$, and $M_{1} \neq 0$, then as $n \to \infty$,
      \begin{align*}
        \tfrac{1}{\sqrt{n\log(n)}}
        \sum_{k=0}^{n-1} X \circ B^{k} 
        &\xrightarrow{dist} 
        N(0,C_{0}+C_{1}).
      \end{align*}
  \end{enumerate}
\end{thm}
%>%
%<% Prargraph 4
Note that by manipulating the values of $M_{0}$ and $M_{1}$ in the third limit theorem above on can obtain stable distributions with any skewness parameter $b \in [-1,1]$.\par
%>%
%<% Paragraph 5
The choices of parameter ranges in the last three limit theorems above are motivated by the following lemma.
\par
%>%
%<% Lemma: Finite Variance Conditions
\begin{lem}[Finite Variance Conditions]
  \label{lem:Base Observable is L2}
  Suppose that $X\colon \left[ 0,1 \right]^{2} \to \R$ is $\gamma$-H\"older for some $\gamma\in(0,1]$. If for $j =0$ one of the conditions below is satisfied, and similarly for $j = 1$ one of the conditions below is satisfied, then $\xi \in L^{2}$.
  \begin{enumerate}[i.]
    \item $\alpha_{j}<1$,
    \item $M_{j} = 0$ and $\alpha_{j}=1$,
    \item $M_{j}= 0$, $1< \alpha<3$, and $\gamma>\tfrac{\alpha-1}{2}$,
  \end{enumerate}
\end{lem}
\begin{proof}
  See \Cref{sec:Finite Variance Conditions}
\end{proof}
%>%
%<% Paragraph 6
The proof of this \Cref{thm:Main Limit Theorem} is an application of \cite{Gouezel2004-Intermittent} Theorem 2.1. For the convenience of the reader we reproduce the theorem here. We have modified the notation slightly to fit our setting.
\par
%>%
%<% Theorem: Gouezel Limit Theorems
\begin{thm}[Theorem 2.1 from \cite{Gouezel2004-Intermittent}]
  \label{thm:Gouezel2}
  Let $\Bs$ be a Banach space and $R_{n}\in Hom(\Bs,\Bs)$ be operators on $\Bs$ with $\|R_{n}\|\le r_{n}$ for a sequence $r_{n}$ such that $a_{n} = \sum_{k>n} r_{k}$ is summable. Write $R(z) = \sum R_{n} z^{n}$ for $z\in \Dbar$. Assume that $1$ is a simple isolated eigenvalue of $R(1)$ and that $I-R(z)$ is invertible for $z\in \Dbar - \left\{ 1 \right\}$. Let $P$ denote the spectral projection of $R(1)$ for the eigenvalue $1$, and assume that $PR'(1)P = \mu P$ for some $\mu>0$. Let $R_{n}(t)$ be an operator depending on $t\in[-\delta_{0},\delta_{0}]$, continuous at $t=0$ with $R_{n}(0) = R_{n}$ and $\|R_{n}(t)\|\le Cr_{n}$ for all $t \in [-\delta_{0},\delta_{0}]$, for some constant $C>0$. For $z\in \Dbar$ and $t \in [-\delta_{0}, \delta_{0}]$ write 
  \[
    R(z,t) = \sum_{n=1}^{\infty} z^{n} R_{n}(t).
  \]
  This is a continuous perturbation of $R(z)$. For $t$ small and $z$ close to $1$, $R(z,t)$ is close to $R(1)$, whence it admits an eigenvalue $\chi(z,t)$ close to $1$. Assume that $\chi(1,t) = 1 - (c+o(1))M(\abs{t})$ for $c \in \C$ with $Re(c)>0$, and some continuous function $M\colon \R_{+} \to \R_{+}$ vanishing only at $0$. Then
  \begin{enumerate}
    \item There exists $\epsilon_{0}>0$ such that for all $\abs{t}<\epsilon_{0}$, $I-R(z,t)$ is invertible for all $z \in \D$. We can write $(I-R(z,t))^{-1} = \sum T_{n,t}z^{n}$.
    \item Furthermore, there exist functions $\epsilon(t)$ and $\delta(n)$ tending to $0$ when $t \to \infty$ and $n \to \infty$ such that for all $\abs{t}<\epsilon_{0}$, for all $n \in \N^{*}$, we have 
      \[
        \left\|
        T_{n,t} = \tfrac{1}{\mu}\left( 1-\tfrac{c}{\mu}M\left( \abs{t} \right)^{n}P \right) \le \epsilon(t) + \delta(n).
        \right\|
      \]
  \end{enumerate}
\end{thm}
%>%
%<% Paragraph 7
Before we can apply \Cref{thm:Gouezel2} we must define the operators $R_{n}(t)$. First, we define an observable $\xi \colon \Lambda \to \R$ derived from the observable $X$ as follows,
\begin{equation}
  \label{eqn:Base Observable Definition}
  \xi(x,y) = \sum_{k=0}^{r(x)-1} (X\circ B^{k})(x,y).
\end{equation}
For all $t \in \R$ and $\eta \in \ULip$, let 
\begin{equation}
  \label{eqn:Perturbed Exact Return Operators}
  R_{n}(t)\eta = R_n\left[ \exp\left( it\xi \right)\eta \right].
\end{equation}
\par
%>%
%<% Paragraph 8
Note that the hypotheses of \Cref{thm:Gouezel2} that pertain to the unperturbed operators have already been verified in \Cref{prop:Convergence and Renewal Equation,prop:Spectrial Gap and Aperiodicity}. Recall from the proof of \Cref{thm:Main Correlation} that $\mu = \frac{1}{\Leb(\Lambda)}$. The following two propositions verify the remaining hypotheses of \Cref{thm:Gouezel2} that pertain to the perturbed operator.
\par
%>%
%<% Proposition: Convergence and Continuity of Perturbations
\begin{prop}[Convergence and Continuity of Perturbations]
  \label{prop:Convergence and Continuity of Perturbations}\leavevmode
  \begin{itemize}
    \item The operators $R_{n}(t)$ are continuous at $t=0$.
    \item As $t \to 0$, $\|R(z,t) - R(z,0)\| = O \left( \abs{t} \right)$.
    \item There exists $\delta_{0}>0$ and $C>0$ such that for all $t \in [-\delta_{0}, \delta_{0}]$ and $n\in\N^{*}$, $\|R_{n}\| \le Cr_{n}$.
  \end{itemize}
\end{prop}
\begin{proof}
  See \Cref{sec:Convergence and Continuity of Purturbations}.
\end{proof}
%>%
%<% Proposition: Expansion of Dominant Eigenvalue
\begin{prop}[Expansion of Dominant Eigenvalue]
  \label{prop:Expansion of Dominant Eigenvalue}
  Let $\chi(t)$ denote the eigenvalue near $1$ of the operator $R(1,t)$ for small $t$.
  Suppose that $X\colon \left[ 0,1 \right]^{2} \to \R$ is $\gamma$-H\"older for some $\gamma\in(0,1]$ and $\int_{[0,1]^{2}} X \,d \Leb = 0$. 
  \begin{enumerate}[i.]
    \item Suppose that\footnote{In particular if the hypotheses of \Cref{lem:Base Observable is L2} are satisfied} $\xi \in L^{2}$ and that $\xi$ is not a coboundary. Then $\sigma^{2} = \int_{\Lambda} \abs{\xi}^{2} \, d \lambda + 2 \sum_{k=1}^{\infty} \int_{\Lambda} \xi\circ T^{k} \, \xi \, d \lambda$ converges, $\sigma^{2}>0$, and as $t\to 0$,
      \begin{align*}
        \chi(t) 
        &\sim 1 - \left(\tfrac{1}{2}\sigma^{2} + o(1)\right)t^{2},
      \end{align*}
    \item Suppose that $\alpha_{0} > \alpha_{1}$, $\alpha_{0}>1$, and $M_{0} >0$. 
      Let $p = 1+\frac{1}{\alpha_{0}}$, $a = C_{0}\Gamma(1-p)\cos\left(\frac{p\pi}{2}\right)$, and $b = 1$. As $t \to 0$,
      \begin{align*}
        \chi(t)
        &=
        1 - \left( a\left(1-ib\sign(t)\tan\left( \frac{p\pi}{2} \right)\right) + o(1) \right)\abs{t}^{p}
      \end{align*}
    \item Suppose that $\alpha_{0} = \alpha_{1}=:\alpha$, $\alpha>1$, $M_{0}>0$ and $M_{1}<0$. 
      Let $p = 1+\frac{1}{\alpha}$, $a = \left(C_{0} + C_{1}\right)\Gamma(1-p)\cos\left(\frac{p\pi}{2}\right)$, and $b = \frac{C_{0}-C_{1}}{C_{0}+C_{2}}$. As $t \to 0$,
      \begin{align*}
        \chi(t)
        &=
        1 - \left( a\left(1-ib\sign(t)\tan\left( \frac{p\pi}{2} \right)\right) + o(1) \right)\abs{t}^{p}
      \end{align*}
    \item Suppose that $\alpha_{0} = \alpha_{1} = 1$, $M_{0} \neq 0$, and $M_{1} \neq 0$. As $t \to 0$,
      \begin{align*}
        \chi(t)
        &=
        1+\left( \tfrac{1}{2}\left(C_{0}+C_{1}\right) +o(1) \right)\abs{t}^{2}\log\abs{t}.
      \end{align*}
  \end{enumerate}
\end{prop}
\begin{proof}
  See \Cref{sec:Expansion of Dominant Eigenvalue}
\end{proof}
%>%
%<% Proof: Main Limit Theorems
\begin{proof}[Proof of \Cref{thm:Main Limit Theorem}]
  The results follow from arguments similar to those presented in \cite{Gouezel2004-Intermittent} Sections 4.3 and 4.4. For the proof of ($iv$) it is worth noting that
  \begin{align*}
    \chi\left( \tfrac{t}{\sqrt{n\log(n)}} \right) 
    &=
    1+(C_{0}+C_{1})t^{2}\tfrac{1}{n}\left[\tfrac{\log(t)}{\log(n)} - \tfrac{\log\left( \log(n) \right)}{2\log(n)} -\tfrac{1}{2}  \right]\\
    &=
    1-\tfrac{1}{2}(C_{0}+C_{1})t^{2}\tfrac{1}{n}\left[1-o(1)\right]\\
    &\sim
    1-\tfrac{1}{2}(C_{0}+C_{1})t^{2}\tfrac{1}{n}.
  \end{align*}
  Therefore,
  \[
    \lim_{n\to \infty}
    \left[
      \chi\left( t \sqrt{\tfrac{\log(n)}{n}} \right) 
    \right]^{n}
    =
    \exp\left( -\tfrac{1}{2}(C_{0}+C_{1})t^{2} \right).
  \]
\end{proof}
%>%
%>%
%>%
%<% Section: Technical Results
\section{Technical Results}
\label{sec:Technical Results}
In this section we verify \Cref{prop:Convergence and Renewal Equation,prop:Spectrial Gap and Aperiodicity,prop:Convergence and Continuity of Perturbations,prop:Expansion of Dominant Eigenvalue}.
\subsection{Compact Embedding}
\label{sec:Compactness}
In this section we will show that $\Bs$ is compactly embedded into $\Bw$. This is necessary for  us to apply Hennion's theorem in \Cref{sec:Essential Spectrum} to show that the operators $R(z)$ acting on $\Bs$ are quasi-compact.
\begin{prop}
  \label{prop:Compact Embedding}
  The inclusion of $\Bs$ into $\Bw$ is a compact embedding.
\end{prop}
\begin{lem}
  \label{lem:Abstract Compactness}
    Let $U$ be a linear subspace of a Banach space with norm $\|\cdot\|$. Suppose that for all $\epsilon>0$ there exist a finite set of linear functionals $\left\{\alpha_{1},\dots,\alpha_{k}  \right\}$ defined on $U$ such that for all $\eta \in U$,
  \[
    \|\eta\| \le \max_{1\le i \le k} \abs{\alpha_{i}\left( \eta \right)} + \epsilon.
  \]
  Then $U$ is a compactly embedded subspace.
\end{lem}
\begin{proof}
  Let $\alpha \colon U \to \R^{k}$, be the linear mapping with coordinate functions $\alpha_{1},\dots,\alpha_{k}$. We will view $\R^{k}$ as a normed linear space equipped with the max-norm. By the supposed bound, $\alpha$ has operator norm $1$. Let $U_{1}$ denote the unit ball of $U$ and note that $\alpha(U_1)$ is a subset of the unit ball of $\R^{k}$. Fix $\epsilon>0$ and let $\left\{ V_{1},\dots,V_{j} \right\}$ be a finite cover of the unit ball of $\R^k$ by balls of radius $\epsilon$.\par
  
  The collection $\left\{ \alpha^{-1}V_{1},\dots,\alpha^{-1}V_{k} \right\}$ is a cover of $U_{1}$. In fact, this collection is a cover by sets of diameter at most $3 \epsilon$. To verify this, fix $j \in \left\{ 1,\dots,k \right\}$ and suppose that $\eta$ and $\nu$ are in $\alpha^{-1}V_{j}$. Since $V_{j}$ is a max-norm ball of radius $\epsilon$ we have 
  \[
    \|\alpha(\eta - \nu)\|_{\max}
    =
    \|\alpha(\eta) - \alpha(\nu)\|_{\max}
    \le 
    \mathrm{diam}(V_{j})
    \le
    2\epsilon
  \]
  By the supposed bound, we obtain
  \[
    \|\eta - \nu\| 
    \le 
    \max_{1\le i \le k} \abs{\alpha_{i}(\eta-\nu)} + \epsilon 
    =
    \|\alpha(\eta - \nu)\|_{\max} + \epsilon
    \le
    3\epsilon.
  \]
  For each $j \in \left\{ 1,\dots,k \right\}$, select $\eta_{j} \in \alpha^{-1}V_{j}$ and let $B_{j}$ be the open $\|\cdot\|$-ball of radius $4\epsilon$ centered at $\eta_{j}$. By the choice of radius, we see that $\alpha^{-1}V_{j} \subset B_{j}$. Therefore, $\left\{ B_{1},\dots,B_{k} \right\}$ is an open cover of $U_{1}$ by balls of radius $4\epsilon$. Since $\epsilon>0$ was arbitrary, we conclude that $U_{1}$ is totally bounded with respect to the metric induced by $\|\cdot\|$. Therefore, $U$ is a compactly embedded subspace of the Banach space.
\end{proof}
\begin{lem}
  The space $\Tl$ is compactly embedded into $\Ta$.
\end{lem}
\begin{proof}
  Recall that $\Tl$ and $\Ta$ are respectively Lipschitz and H\"older functions on $[0,1]$. The result is classical.
\end{proof}
\begin{proof}[Proof of \Cref{prop:Compact Embedding}]
  Fix $\epsilon >0$. Chose a set $\left\{ w_{1},\dots,w_{m} \right\}\subset [p,q]$ that is $\frac{\epsilon}{2}$-dense. Chose $\left\{ \xi_{1},\dots,\xi_{n} \right\} \subset \Tl$ that is $\frac{\epsilon}{2}$-dense in the unit ball of $\Tl$ with respect to $\NTa{\cdot}$. Note that for all $j\in \left\{1,\dots,n\right\}$, $\NTa{\xi_{j}}\le\NTl{\xi_{j}}\le1$.\par

  For $\eta \in \ULip$ with $\NBs{\eta}<\infty$, $i\in\left\{1,\dots,m\right\}$, and $j \in \left\{1,\dots,n\right\}$, define $\alpha_{ij}(\eta) = \int_{0}^{1} \eta(w_{i},y)\, \xi_{j}\, dy$. These linear functionals are in $\Bw^*$ with norm at most 1. Therefore, the functionals $\alpha_{ij}$ are defined on the linear subspace $\Bs\subset\Bw$. \par

  For all $\eta \in \ULip$ with $\NBs{\eta}<\infty$, $x \in [p,q]$, $\psi \in \Tl$, $i\in\{1,\dots,m\}$, and $j \in \left\{ 1,\dots,n \right\}$,
  \begin{align*}
    \int_{0}^{1} \eta(x,y) \, \psi(y) \, dy
    &=
    \int_{0}^{1} \eta(x,y) \, \left[\psi-\xi_{j}\right](y) \, dy\\
    &+
    \int_{0}^{1} \left[\eta(x,y)-\eta(w_i,y)\right] \, \xi_{j}(y) \, dy\\
    &+
    \int_{0}^{1} \eta(x_i,y) \, \xi_{j}(y) \, dy\\
    &\le\NBs{\eta}\NTa{\psi - \xi_{j}}
    +\MBs{\eta} \abs{x-w_{i}}
    +\abs{\alpha_{ij}(\eta)}
  \end{align*}
  By selecting $i$ so that $w_i$ and $x$ are close, and selecting $j$ so that $\psi$ and $\xi_{j}$ are close we see that 
  \[
    \int_{0}^{1} \eta(x,y) \, \psi(y) \, dy \le \abs{\alpha_{ij}(\eta)} + \epsilon.
  \]
  Take a maximum over $i\in\left\{ 1,\dots,m \right\}$ and $j\in\left\{ 1,\dots,n \right\}$ on the right hand side of the inequality above, and a supremum over $x\in[p,q]$ and $\psi\in \Tl$ with $\NTl{\psi}\le 1$ on the left hand side to obtain
  \begin{equation}
    \label{eqn:Finite Approximation of the Norm}
    \NBw{\eta} \le \max_{i,j} \abs{\alpha_{ij}\left( \eta \right)} +\epsilon.
  \end{equation}
  Since the set of $\eta\in\ULip$ with $\NBs{\eta}<\infty$ is dense in $\Bs$, the bound above extends to all $\eta \in \Bs$.
  We have shown that for all $\epsilon>0$ there exists a finite collection of bounded linear functionals $\alpha_{ij} \in \Bw^{*}$, such that for all $\eta \in \Bs$, we have \Cref{eqn:Finite Approximation of the Norm}. By \Cref{lem:Abstract Compactness}, the inclusion of $\Bs$ into $\Bw$ is compact.
\end{proof}
\subsection{Essential Spectrum}
\label{sec:Essential Spectrum}
In this section we show that the operators $R(z)$ are quasi-compact for $\abs{z}\le 1$ (see \Cref{prop:Quasi-Compactness}). This is the first step toward proving the finer spectral properties of $R(z)$ obtained in \Cref{prop:Spectrial Gap and Aperiodicity}. 
In the process of proving quasi-compactness we will produce bounds on the operators $R_{k}$ (see \Cref{lem:Basic Norm Bounds}) that will be sufficient to prove \Cref{prop:Convergence and Renewal Equation}.\par

Recall that $\beta$ is defined in \Cref{eqn:beta Definition} and is related to the unstable expansion of the map $T$ and that $a\in (0,1]$ is a parameter related to the norm on $\Ta$ and is fixed in \Cref{sec:Adapted Banach Spaces}.
\begin{prop}[Quasi-Compactness]
\label{prop:Quasi-Compactness}
For each $\abs{z}\le 1$ the operator $R(z)\colon \Bs  \to  \Bs $ is quasi-compact with spectral radius less than or equal to $|z|$ and essential spectral radius less than or equal to $\beta^{a}\abs{z}$.
\end{prop}

The proof of \Cref{prop:Quasi-Compactness} is an application of the following theorem of Hennion.

\begin{thm}[Hennion \cite{Hennion-2001} via Liverani \cite{Liverani-2004}]
\label{thm:Hennion}
If $\Bs\subseteq \Bw$ are Banach spaces with norms $\NBs{\cdot}$ and $\NBw{\cdot}$ respectively, such that $\NBw{\cdot}\le \NBs{\cdot}$, and $L:\Bs \to \Bs$ is a bounded linear operator such that:
\begin{enumerate}
    \item $L:\Bs \to \Bw$ is a compact operator;
    \item There exists $\theta, A,B,C>0$ such that for all $k\in \N$ there exists $M_k>0$ such that for all $f\in \Bs$, we have
        \begin{enumerate}
          \item $\displaystyle \NBw{L^{k}f} \le CM_{k}\NBw{f}$,
          \item $\displaystyle \NBs{L^{k}f} \le A\theta^{k}\NBs{f}+ BM_{k} \NBw{f}$.
        \end{enumerate}
\end{enumerate}
Then $L\colon \Bs  \to  \Bs $ is quasi compact with essential spectral radius less than or equal to $\theta$.
We will refer to the second inequality above as the \emph{Lasota-Yorke inequality}.
\end{thm}
In order to apply \Cref{thm:Hennion} to obtain \Cref{prop:Quasi-Compactness} we need the following inequalities. 
\begin{prop}[Uniform Lasota-Yorke Inequality]
\label{prop:Uniform Lasota-Yorke Inequality}
For all $\eta \in \ULip$ and $k\ge1$
\begin{align}
  \label{eqn:Lasota-Yorke Strong-Strong Inequality}
  \NBs{R(z)^{k}\eta} 
  &\le 
  \left[\kappa+1\right]
  \abs{z}^{k}
  \NBs{\eta},\\
  \label{eqn:Lasota-Yorke Weak-Weak Inequality}
  \NBw{R(z)^{k}\eta} 
  &\le 
  \left[\kappa+1\right]:
  \abs{z}^{k}
  \NBw{\eta},\\
  \label{eqn:Lasota-Yorke Main Inequalty}
  \NBs{R(z)^{k}\eta} 
  &\le 
  \left[\kappa+1\right]
  \abs{z}^{k}
  \left[
    \left( \beta^a \right)^{k} \NBs{\eta} 
    + 
    \left( \kappa +\tau +1 \right)\NBw{\eta}
  \right].
\end{align}
\end{prop}
\begin{proof}
	This proposition is proved in \Cref{sec:Proof of Lasota-Yorke Inequality}
\end{proof}

We are now in a position to prove \Cref{prop:Quasi-Compactness}.
\begin{proof}[Proof of \Cref{prop:Quasi-Compactness}]
  As is discussed in \Cref{sec:Adapted Banach Spaces}, $\NBw{\cdot}\le\NBs{\cdot}$ and as a result $\Bs$ can be viewed as a subset of $\Bw$. The operator $R(z)\colon \Bs  \to  \Bs $ is bounded by the first inequality from \Cref{prop:Uniform Lasota-Yorke Inequality}.\par
  
  The compactness of $R(z)\colon\Bs \to \Bw$ follows from \Cref{prop:Compact Embedding}, which states that the inclusion of $\Bs$ into $\Bw$ is compact. The transformation $R(z)\colon \Bs \to \Bw$ is the composition of the bounded operator $R(z)\colon \Bs \to \Bs$ followed by the compact inclusion from $\Bs$ into $\Bw$, and thus is compact. Therefore, the first hypothesis of \Cref{thm:Hennion} is satisfied.\par

  The second hypothesis of \Cref{thm:Hennion} is verified by the second and third inequalities obtained in \Cref{prop:Uniform Lasota-Yorke Inequality}. By inspecting the third inequality from \Cref{prop:Uniform Lasota-Yorke Inequality} we see that $\theta = \abs{z}\beta^{a}$ and we have verified the claimed bound on the essential spectral radius of $R(z)$.\par

  The claimed bound on the spectral radius of $R(z)$ follows from the first inequality in \Cref{prop:Uniform Lasota-Yorke Inequality} and the Gelfand spectral radius formula.
\end{proof}
\subsection{Basic Norm Bounds}
The following lemma provides the key estimates required to prove \Cref{prop:Uniform Lasota-Yorke Inequality,prop:Convergence and Renewal Equation}.
\begin{lem}[Basic Norm Bounds]
\label{lem:Basic Norm Bounds}
For all $k\ge 1$, $n\ge 1$, and $\eta \in \ULip$,
\begin{align}
  \label{eqn:Basic Weak-Weak Bound}
  \NBw{R^{(k)}_{n}\eta} 
  &\le 
  \left[ \kappa+1 \right]
  \lambda\left\{r^{(k)}=n\right\}
  \NBw{\eta},\\
  \label{eqn:Basic Strong-Strong Bound}
  \Ns{R^{(k)}_{n}\eta} 
  &\le 
  \left[ \kappa+1 \right]
  \lambda\left\{r^{(k)}=n\right\}
  \Ns{\eta},\\
  \label{eqn:Basic Modulus Bound}
  \MBs{R^{(k)}_{n}\eta} 
  &\le 
  \left[ \kappa+1 \right]
  \lambda\left\{r^{(k)}=n\right\}
  \left[
    \beta^{k}\MBs{\eta} 
    + 
    \left( \tau + \kappa \right)
    \NBw{\eta}
  \right],\\
  \label{eqn:Basic Stable Lasota-Yorke Bound}
  \Ns{R^{(k)}_{n}\eta}
  &\le
  \left[ \kappa+1 \right]
  \lambda\left\{r^{(k)}=n\right\}
  \left[ 2\left( \beta^a \right)^k\Ns{\eta} + \NBw{\eta} \right],\\
  \label{eqn:Basic Strong Lasota-Yorke Bound}
  \NBs{R^{(k)}_{n}\eta} 
  &\le 
  \left[ \kappa+1 \right]
  \lambda\left\{r^{(k)}=n\right\}
  \left[ 
    2\left( \beta^a \right)^{k} \NBs{\eta} + (\kappa +\tau+1)\NBw{\eta}
  \right].
\end{align}
\end{lem}

\begin{proof}
  Recall the definitions of norms from \Cref{sec:Adapted Banach Spaces} and the definition of the operators $R_{n}^{(k)}$ from \Cref{eqn:Exact Return Operators}. All of the quantities that we wish to bound are defined through integrals of the form
  \begin{align}
    \label{eqn:Basic Bounds Integral Identity}
      \int_{0}^{1} 
        R^{(k)}_{n} \eta(x,y) \, 
        \psi(y)\, 
      dy
      &=
      \int_{0}^{1}
      T_{*}^{k}\left( \eta \indf{\left\{r^{(k)}=n \right\}} \right)(x,y)\,
      \psi(y)\,
      dy\\
      &=
      \int_{0}^{1}
      T_{*}^{k}\eta(x,y)\,
      \psi(y)\,
      \indf{T^{k}\left\{ r^{(k)}=n \right\}}\,
      dy.\notag
  \end{align}

  The set $ T^{k}\left\{ r^{(k)}=n \right\}$ is a countable collection of horizontal strips. Similarly the set $\left\{ r^{(k)}=n \right\}$ is a countable collection of vertical strips. That is, there exists a countable collection of intervals $\left\{ I_{j}\subseteq [p,q]:j\in\Z^+ \right\}$ such that 
  \[
    \left\{ r^{(k)}=n \right\} = \bigcup_{j\in\Z^+} I_{j}\times [0,1].
  \]
  It will be convenient to define $V_{j} = I_{j} \times [0,1]$. Note that for all $j \in \Z^{+}$, $T^{k}V_{j}$ is a horizontal strip and $T^{k}|_{V_{j}}$ is $C^1$. In the horizontal coordinate, $u^{k}|_{I_{j}}\colon I_{j} \to [p,q]$ is a $C^2$ bijection.\par

  The integration in \Cref{eqn:Basic Bounds Integral Identity} is over the set $T^{k}\left\{ r^{(k)}=n \right\} \cap \left( \left\{ x \right\}\times[0,1] \right)$. By the comments of the previous paragraph, for each $j \in \Z^+$, there exists $s_{j}\in I_{j}$ so that $u^{k}(s_{j}) = x$. Therefore, the preimage of the region of integration under the map $T^{k}$ is the countable union of full vertical lines as follows,
  \begin{align*}
    T^{-k}\left(  T^{k}\left\{ r^{(k)}=n \right\} \cap \left( \left\{ x \right\}\times[0,1] \right)\right)
    &=
    \left\{ r^{(k)}=n \right\} \cap T^{-k}\left( \left\{ x \right\}\times[0,1] \right)\\
    &=
    \bigcup_{j\in\Z^+} \left\{ s_{j} \right\}\times [0,1].
  \end{align*}

  Suppose that $(s_{j},t) \in V_{j}$ and that $(x,y) = T^{k}(s_{j},t) = \left( u^{k}(s_{j}), v_{s_{j}}^{(k)}(t) \right)$. Let $\eta \in \ULip$ and $\psi \in \Tl$ or $\Ta$. Suppose that $n \in \Z^+$ and $k\in\Z^+$ are numbers such that the set $\left\{ r^{(k)}=n \right\}$ is non-empty. An elementary change of variables on each line segment $\{s_{j}\}\times[0,1]$ shows that
    \begin{align*}
      \int_{0}^{1} 
        R^{(k)}_{n} \eta(x,y) \, 
        \psi(y)\, 
      dy 
      = 
      \sum_{j = 1}^{\infty}
      \int_{0}^{1} 
      \eta(s_{j},t) \, 
      \psi\left(v^{(k)}(s_{j},t)\right)\, 
      \partial_{t}v^{(k)}(s_{j},t)\,
      dt,
    \end{align*}
    where we have written $v^{(k)}(s,t)$ instead of $v^{(k)}_{s}(t)$, this helps to avoid nested subscripts and reduces confusion when taking partial derivatives of the map $v^{(k)}\colon \Lambda \to [0,1]$.\par

  Recall from \Cref{sec:Induced Map} that $T$ preserves the measure $\lambda$ on $\Lambda$, which is a probability measure obtained by restricting Lebesgue measure to $\Lambda$. From this we deduce that, $\lambda$ almost everywhere the Jacobian determinant of $T$ is $1$. Calculating the Jacobian determinant of $T^k$ using \Cref{eqn:T-skew} we obtain the identity below. 
    \begin{equation}
      \label{eqn:Basic Bound Jacobian Identity}
      \partial_{y} v^{(k)}(x,y) \, Du^{k}(x) = 1
    \end{equation}
    It follows that for $\mu$ almost every $x$,
    \begin{align}
      \label{eqn:Basic Bound Change of Variable}
      \int_{0}^{1} 
        R^{(k)}_{n} \eta(x,y) \, 
        \psi(y)\, 
      dy 
      = 
      \sum_{j = 1}^{\infty}
      \int_{0}^{1} 
      \eta(s_{j},t) \, 
      \psi\left(v^{(k)}(s_{j},t)\right)\, 
      \left[ Du^{k}\left( s_{j} \right) \right]^{-1}\,
      dt.
    \end{align}
  
    We apply the definitions of $\Ns{\cdot}$ and $\NBw{\cdot}$ (see \Cref{eqn:Weak Norm Defintion,eqn:Stable Norm Defintion}) to provide the following preliminary bounds.
  If $\psi \in \Tl$, then 
  \begin{align}
    \label{eqn:Basic Bound Prebound-1}
      \int_{0}^{1} 
        R^{(k)}_{n} \eta(x,y) \, 
        \psi(y)\, 
      dy 
      &\le 
      \sum_{j\in\Z^+}
      \left[ Du^{k}\left( s_{j} \right) \right]^{-1}
      \NBw{\eta} \NTl{\psi\circ v_{s_{j}}^{(k)}}.
  \end{align}
  If $\psi \in \Ta$, then 
  \begin{align}
    \label{eqn:Basic Bound Prebound-2}
      \int_{0}^{1} 
        R^{(k)}_{n} \eta(x,y) \, 
        \psi(y)\, 
      dy 
      &\le 
      \sum_{j = 1}^{\infty}
      \left[ Du^{k}\left( s_{j} \right) \right]^{-1}
      \Ns{\eta} \NTa{\psi\circ v_{s_{j}}^{(k)}}.
  \end{align}
  
  In order to bound $\MBs{R^{(k)}_{n}\eta}$ we must consider integrals of the form
  \[
      \int_{0}^{1} 
        \left[
          R^{(k)}_{n} \eta(x,y) -
          R^{(k)}_{n} \eta(w,y)
        \right]\,
        \psi(y)\, 
      dy,
  \]
  where $x,w \in [p,q]$ are fixed horizontal coordinates and $\psi \in \Tl$ is a test function. If we separate the integral above by linearity and apply \Cref{eqn:Basic Bound Change of Variable} to each term we obtain
  \begin{align*}
    \int_{0}^{1} 
    \left[
      R^{(k)}_{n} \eta(x,y) -
      R^{(k)}_{n} \eta(w,y)
    \right]\,
    \psi(y)\, 
    dy 
    &=
    \sum_{j = 1}^{\infty}
    \left[ Du^{k}(s_{j}(x)) \right]^{-1}
    \int_{0}^{1} 
    \eta\left(s_{j}(x),t\right) \, 
    \psi\left(v^{(k)}(s_{j}(x),t)\right)\, 
    dt\\ 
    &-
    \sum_{j = 1}^{\infty}
    \left[ Du^{k}(s_{j}(w)) \right]^{-1}
    \int_{0}^{1} 
    \eta\left(s_{j}(w),t\right) \, 
    \psi\left(v^{(k)}(s_{j}(w),t)\right)\, 
    dt\\ 
  \end{align*}
  where $s_{j}(x)$ and $s_{j}(w)$ are points in $I_{j}$ such that $u^{k}\left( s_{j}(x) \right)=x$ and $u^{k}\left( s_{j}(w) \right)=w$. We will fix $j$ and expand $j$-th term of right hand side of the identity above following the pattern
  \[
    A_x B_x C_x 
    - 
    A_w B_w C_w 
    =
    \left( A_x - A_w \right) B_x C_x
    +
    A_w \left( B_x - B_w \right) C_x
    +
    A_w B_w \left( C_x - C_w \right) 
  \]
  and apply the definitions of $\Ns{\cdot}$ and $\MBs{\cdot}$ (see \Cref{eqn:Unstable Modulus Defintion,eqn:Strong Norm Defintion}) to obtain the preliminary bound below.
  \begin{align}
    \label{eqn:Basic Bound Prebound-3}
    \int_{0}^{1} 
    \left[
      R^{(k)}_{n} \eta(x,y) -
      R^{(k)}_{n} \eta(w,y)
    \right]\,
    \psi(y)\, 
    dy 
    &\le
      \sum_{j = 1}^{\infty}
    \left[ Du^k(s_{j}(w)) \right]^{-1}
    \abs{
      1
      - 
      \frac{
      Du^k(s_{j}(w))
      }{
      Du^k(s_{j}(x))
      }
    }
    \NBw{\eta} 
    \NTl{\psi\circ v^{(k)}_{s_{j}(x)}}\notag\\
    &+
      \sum_{j = 1}^{\infty}
    \left[ Du^k(s_{j}(w)) \right]^{-1}
    \MBs{\eta}
    \abs{s_{j}(x)-s_{j}(w)}
    \NTl{\psi\circ v^{(k)}_{s_{j}(x)}}\notag\\
    &+
      \sum_{j = 1}^{\infty}
    \left[ Du^k(s_{j}(w)) \right]^{-1}
    \NBw{\eta} 
    Lip(\psi) 
    \sup_{0\le t\le 1} \abs{v^{(k)}_{s_{j}(x)}(t)-v^{(k)}_{s_{j}(w)}(t)}.
  \end{align}

  Having collected the preliminary bounds \Cref{eqn:Basic Bound Prebound-1,eqn:Basic Bound Prebound-2,eqn:Basic Bound Prebound-3} we wish to improve them to bounds that are uniform in the horizontal coordinates $x$ and $w$. To this end we collect the following uniform bounds.\par

    Fix $j \in \Z^{+}$, suppose that $s \in I_{j}$ and $u^{k}(s) = x$. Since $u^{k}$ maps $I_{j}$ onto $[p,q]$,
    \[
      \frac{1}{\mu(I_{j})}
      \int_{I_{j}} 
      Du^{k}(s) \, d\mu(s)
      =
      \frac{1}{\mu(I_{j})}
      \int_{p}^{q} 1 \,d\mu(x)
      =
      \frac{1}{\mu(I_{j})}.
    \]
  Since $u^{k}|_{I_{j}}$ is $C^2$, we can apply the Integral Mean Value Theorem to conclude that there exists $\theta \in I_{j}$ such that 
    \[
      Du^{k}(\theta) = \frac{1}{\mu\left( I_{j} \right)}.
    \]
    Since $s, \theta \in I_{j} \subset \left[ r^{(k)} = n \right]$ we may apply \Cref{eqn:Distortion Bound} to obtain
  \begin{align*}
    \abs{
      \frac{
        Du^{k}(\theta)
      }{
        Du^{k}(s)
      }
      -1
    }
    &\le
    \kappa \abs{u^{k}\left(s \right)-u^{k}\left( \theta \right)}
    \le
    \kappa
  \end{align*}
  Since $u^{k}|_{I_{j}}$ is order preserving, we have $Du^{k}(s) \ge 0$. In combination with the last two displayed equations this yields,
  \begin{align}
    \label{eqn:Basic Bound Vertical Contraction}
    0 \le \left[Du^{k}(s)\right]^{-1} \le \left( \kappa +1 \right)\mu\left(I_{j}\right).
  \end{align}
  By another application of \Cref{eqn:Distortion Bound}, for all $w,x\in[p,q]$,
  \begin{align}
    \label{eqn:Basic Bound Vertical Distortion}
    \abs{
      1
      - 
      \frac{
        Du^{k}(s_{j}(w))
      }{
        Du^{k}(s_{j}(x))
      }
    }
    &\le
    \kappa
    \abs{u(s_{j}(x)) - u(s_{j}(w))}
    =
    \kappa\abs{x-w}.
  \end{align}

  Since $v^{(k)}_{s(z)}$ viewed as a map of the interval $[0,1]$ into itself is a contraction for all $z \in [p,q]$, we have 
    \begin{equation}
      \label{eqn:Basic Bound psi Koopman}
      \NTl{\psi\circ v^{(k)}_{s(z)}} \le \NTl{\psi}.
    \end{equation}
    Let us view $x \mapsto s_{j}(x)$ as a local inverse of $u^{k}$ on $I_{j}$ and note that,
    \[
      \partial_{x}\left[ v^{(k)}\left( s_{j}(x),t \right) \right] = \frac{\partial_{x}v^{(k)}(s_{j}(x),t)}{Du^{k}\left( s_{j}(x) \right)}.
    \]
  By \Cref{lem:Skew Bound} we have 
    \[
      \Ninfty{\frac{\partial_{x}v^{(k)}_{x}}{Du^{k}}} \le \tau.
    \]
  From the previous two lines it follows that
  \begin{align}
    \sup_{0\le t\le 1} \abs{v^{(k)}_{s(x)}(t)-v^{(k)}_{s(w)}(t)}
    &\le
    \Ninfty{
      \frac{\partial_{x}v^{(k)}_{x}}{Du^{k}}
    }
    \abs{x - w}.\notag\\
    \label{eqn:Basic Bound Vertical Lipschitz}
    &\le \tau \abs{x-w}
  \end{align}

  Finally, note that $\abs{s_{j}(x)-s_{j}(w)} \le \beta^{k}\abs{x-w}$ by \Cref{eqn:Expansion Bound}.\\

  Having collected the uniform bounds above we refine \Cref{eqn:Basic Bound Prebound-1,eqn:Basic Bound Prebound-2,eqn:Basic Bound Prebound-3} as follows. If $\psi \in \Tl$ and $x \in [p,q]$, then 
  \begin{align*}
      \int_{0}^{1} 
        R^{(k)}_{n} \eta(x,y) \, 
        \psi(y)\, 
      dy 
      &\le 
      \sum_{j\in\Z^+}
      \left( \kappa +1 \right)\mu(I_{j})
      \NBw{\eta} \NTl{\psi}\\
      &=
      \left( \kappa +1  \right) \lambda\left\{ r^{(k)}=n \right\}
      \NBw{\eta} \NTl{\psi}.\\
  \end{align*}
  By taking a supremum over $\psi$ such that $\NTl{\psi}\le 1$ and $x \in [p,q]$ we obtain \Cref{eqn:Basic Weak-Weak Bound}.
  Simmilarly, if $\psi \in \Ta$ and $x \in [p,q]$, then
  \begin{align*}
      \int_{0}^{1} 
        R^{(k)}_{n} \eta(x,y) \, 
        \psi(y)\, 
      dy 
      &\le
      \left( \kappa +1  \right) \lambda\left\{ r^{(k)}=n \right\}
      \Ns{\eta} \NTa{\psi}.\\
  \end{align*}
  By taking a supremum over $\psi$ such that $\NTa{\psi}\le 1$ and $x \in [p,q]$ we obtain \Cref{eqn:Basic Strong-Strong Bound}.\par
  
  If $\psi\in\Tl$ and $w,x \in [p,q]$, then

  \begin{align*}
    \int_{0}^{1} 
    \left[
      R^{(k)}_{n} \eta(x,y) -
      R^{(k)}_{n} \eta(w,y)
    \right]\,
    \psi(y)\, 
    dy 
    &\le
      \left( \kappa+1 \right)\lambda\left\{ r^{(k)}=n \right\}
      \kappa\abs{w-x}
    \NBw{\eta} 
    \NTl{\psi}\notag\\
    &+
      \left( \kappa+1 \right)\lambda\left\{ r^{(k)}=n \right\}
    \MBs{\eta}
    \beta^k\abs{x-w}
    \NTl{\psi}\notag\\
    &+
      \left( \kappa+1 \right)\lambda\left\{ r^{(k)}=n \right\}
    \NBw{\eta} 
    Lip(\psi) 
    \tau
    \abs{x-w}
  \end{align*}
  By taking a supremum over $\psi$ such that $\NTl{\psi}\le 1$ and $w,x \in [p,q]$ we obtain \Cref{eqn:Basic Modulus Bound}.\par

  Finally, we must verify \Cref{eqn:Basic Stable Lasota-Yorke Bound}. Suppose that $\psi \in \Ta$ and $x \in [p,q]$. By \Cref{eqn:Basic Bound Change of Variable},
  \begin{align*}
      \int_{0}^{1} 
        R^{(k)}_{n} \eta(x,y) \, 
        \psi(y)\, 
      dy 
      &= 
      \sum_{j = 1}^{\infty}
      \int_{0}^{1} 
      \eta(s_{j},t) \, 
      \psi\left(v^{(k)}(s_{j},t)\right)\, 
      \left[ Du^{k}\left( s_{j} \right) \right]^{-1}\,
      dt\\
      &= 
      \sum_{j = 1}^{\infty}
      \int_{0}^{1} 
      \eta(s_{j},t) \, 
      \left[
      \psi\left(v^{(k)}(s_{j},t)\right) 
      -
      \psi\left(v^{(k)}(s_{j},0)\right)
    \right]\,
      \left[ Du^{k}\left( s_{j} \right) \right]^{-1}\,
      dt\\
      &+
      \sum_{j = 1}^{\infty}
      \int_{0}^{1} 
      \eta(s_{j},t) \, 
      \psi\left(v^{(k)}(s_{j},0)\right)\, 
      \left[ Du^{k}\left( s_{j} \right) \right]^{-1}\,
      dt\\
      &\le
      \left( \kappa + 1 \right)
      \lambda\left\{ r^{(k)}=n \right\}
      \Ns{\eta}
      \NTa{\Psi_1}\\
      &+
      \left( \kappa + 1 \right)
      \lambda\left\{ r^{(k)}=n \right\}
      \NBw{\eta}
      \NTl{\Psi_{0}}\\
      \Psi_{0}
      &:=
      \psi\left(v^{(k)}(s_{j},0)\right)\\
      \Psi_{1}
      &:=
      \psi\left(v^{(k)}(s_{j},t)\right) 
      -
      \psi\left(v^{(k)}(s_{j},0)\right)
  \end{align*}
  In the third line we have use the fact that $\Psi_{0}$ is constant and hence Lipschitz. An elementary calculation shows that
  \begin{align*}
    \NTl{\Psi_{0}} &\le \NTl{\psi} \le \NTa{\psi}\\
    \NTa{\Psi_{1}} &\le 2(\beta^a)^k \NTa{\psi}
  \end{align*}
  We conclude that, 
  \begin{align*}
      \int_{0}^{1} 
        R^{(k)}_{n} \eta(x,y) \, 
        \psi(y)\, 
      dy 
      &\le
      \left( \kappa +1 \right) 
      \lambda\left\{ r^{(k)}=n \right\}
      \NTa{\psi}
      \left[ 2\left( \beta^a \right)^k\Ns{\eta} + \NBw{\eta} \right].
    \end{align*}
  By taking a supremum over $\psi$ such that $\NTa{\psi}\le 1$ and $x \in [p,q]$ we obtain \Cref{eqn:Basic Stable Lasota-Yorke Bound}.\par

  Lastly \Cref{eqn:Basic Strong Lasota-Yorke Bound} follows by adding \Cref{eqn:Basic Modulus Bound} and \Cref{eqn:Basic Stable Lasota-Yorke Bound} and noting that $\beta^{k} <\beta^{ak}$.
\end{proof}

\subsection{Proof of Convergence and Renewal Equation}
\label{sec:Convergence and Renewal Equation}
In this section we use \Cref{lem:Basic Norm Bounds} to prove \Cref{prop:Convergence and Renewal Equation}.
\paragraph{Claim 1} For all $n \ge 1$, the operators $B_{n}$ and $R_{n}$ are bounded on $\Bs$. 
      \begin{proof}
        Boundedness of the $R_n$ is given by \Cref{eqn:Basic Strong-Strong Bound} from \Cref{lem:Basic Norm Bounds} with $k=1$.
        Note that $B_{n} = \sum_{k=1}^{n} R_{n}^{(k)}$, and that the collection of sets $\left\{ \left\{r^{(k)}=n\right\}:k=1,\dots,n \right\}$ are disjoint. It follows from \Cref{eqn:Basic Strong-Strong Bound} that 
          \begin{align*}
            \NBs{B_{n}\eta} 
            &\le 
            \sum_{k=1}^{n} \NBs{R_{n}^{(k)} \eta} 
            \le 
            \sum_{k=1}^{n} \left[ \kappa+1 \right]\lambda\left[ r^{(k)}=n \right] \NBs{\eta}\\
            &\le 
            \left[ \kappa +1 \right] \NBs{\eta}.
          \end{align*}
      Therefore, the operators $B_{n}$ are bounded on $\Bs$.
      \end{proof}
  \paragraph{Claim 2} For all $z$ in the open unit disk of $\C$, the operators $B(z)$ and $R(z)$ converge in $Hom\left( \Bs,\Bs \right)$ and satisfy $B(z) = \left( I-R(z) \right)^{-1}$.
    \begin{proof}
      Since the operators $R_{n}$ and $B_{n}$ are uniformly bounded in $n$, $R(z)$ and $B(z)$ converge for $\abs{z}<1$ as desired.\par

      The IBT $B$ is clearly conservative and non-singular with respect to Lebesgue measure and $\Lambda$ has positive measure. That $B(z)$ and $R(z)$ satisfy $B(z) = \left( I-R(z) \right)^{-1}$ can be verified by applying \cite{Sarig2002} Proposition 1.
    \end{proof}
  \paragraph{Claim 3} The operator $R(z)$ converges in $Hom\left( \Bs, \Bs \right)$ for $z$ in the closed unit disk of $\C$, that is $\sum_{n\ge1} \NBs{R_{n}} < \infty$.
  \begin{proof}
    Note that the sets $[r = n]$ partition $\Lambda$. By \Cref{eqn:Basic Strong-Strong Bound} with $k=1$, for any $\eta \in \ULip$,
    \begin{align*}
      \NBs{R(1)\eta} 
      &\le 
      \sum_{n=1}^{\infty} \NBs{R_{n}\eta}
      \le 
      \left[ \kappa +1  \right]\NBs{\eta} \sum_{n=1}^{\infty} \lambda\left[ r=n \right]\\
      &=
      \left[ \kappa +1  \right]\NBs{\eta} 
    \end{align*}
    So $R(z)$ is bounded on $\Bs$ for all $\abs{z}\le1$.
  \end{proof}
    \paragraph{Claim 4} Let $\alpha = \max\left\{\alpha_{0},\alpha_{1}\right\}$. As $n \to \infty$, $\sum_{k>n} \NBs{R_{k}} = O\left(n^{-\left(1+\frac{1}{\alpha}\right)}\right)$. 
    \begin{proof}
      It follows directly from the definition of $\lambda$ in \Cref{eqn:Base Measure Definition} and the relationship between $p_{n}^{\circ}$, $q_{n}^{\circ}$ and return time outlined in \Cref{eqn:Interval Orbit Structure} that for $n\ge2$, 
      \[
        \lambda[r=n] = \frac{p_{n-1}^{\circ} - p_{n}^{\circ} + q_{n}^{\circ} - q_{n-1}^{\circ}}{\Leb\left( \Lambda \right)}
      \]
      By the asymptotic behavior of $p_{n-1}^{\circ} - p_{n}^{\circ}$ and $q_{n-1}^{\circ}-q_{n}^{\circ}$ described in \Cref{eqn:Period-2 Orbit Increment Asymptotics Interior Right,eqn:Period-2 Orbit Increment Asymptotics Interior Left} we have 
      \[
        \lambda[r=n] \asymp \left(\tfrac{1}{n}\right)^{2+\frac{1}{\alpha}}.
      \]
      By the norm bound obtained in \Cref{eqn:Basic Strong-Strong Bound} with $k=1$, $\left\|R_{n}\right\| = O\left(\lambda[r=n]\right)$, therefore
      \[
        \sum_{k> n}\NBs{R_{k}} =O\left(\left( \frac{1}{n} \right)^{\frac{1}{\alpha} +1}\right).
      \]
    \end{proof}

\subsection{Proof of Lasota-Yorke inequality}
\label{sec:Proof of Lasota-Yorke Inequality}
In this section we will use \Cref{lem:Basic Norm Bounds} to prove \Cref{prop:Uniform Lasota-Yorke Inequality}. 

\begin{proof}[Proof of \Cref{prop:Uniform Lasota-Yorke Inequality}]
We will prove \Cref{eqn:Lasota-Yorke Main Inequalty}. The proofs of the other inequalities are similar.\par

Note that $\min r^{(k)} \ge 2k$ and apply \Cref{lem:Basic Norm Bounds} so that,
\begin{align*}
  \NBs{R(z)^{k}\eta}
  &\le
  \sum_{n=2k}^{\infty}
  \abs{z^{n}}
  \NBs{R_{n}^{(k)}\eta}\\
  &\le
  \abs{z}^{k}
  \sum_{n=2k}^{\infty}
  [\kappa +1] \lambda\left\{ r^{k}=n \right\}\left[2\left( \beta^a \right)^{k} \NBs{\eta} + (\kappa+\tau+1)\NBw{\eta}\right]\\
  &=
  \left[ \kappa+1 \right]\abs{z}^{k}\left[2\left( \beta^a \right)^{k} \NBs{\eta} + (\kappa+\tau+1)\NBw{\eta}\right].
\end{align*}
\end{proof}
Obviously we could have obtained $\abs{z}^{2k}$ as a multiplier in the inequalities above. We opt for the weaker bound as it makes no difference in what follows and is slightly less cumbersome.\par
\subsection{Proof of Spectral Gap and Aperiodicity}
\label{Spectral Gap and Aperiodicity}
In this section we will prove \Cref{prop:Spectrial Gap and Aperiodicity}. The next two lemmas will be useful in the proof.
\begin{lem}
\label{lem:Auxiliary Operator Bound}
If $\abs{z}\le 1$ and $\eta \in \ULip$, then 
\begin{align}
  \label{eqn:Auxiliary Operator Bound}
  \NUL{R(z)\eta} &\le \NUL{\eta}.
\end{align}
\end{lem}
\begin{proof}
We begin by bounding the sup-norm term in $\NUL{\cdot}$,
\begin{align*}
  \Nsup{R(z)\eta}
  &=
  \sup_{(x,y)\in\Lambda}
  \abs{
    \sum_{n=1}^{\infty} z^{n}\, R_{n}\eta(x,y)
  }\\
  &=
  \sup_{(x,y)\in\Lambda}
  \abs{
    \sum_{n=1}^{\infty} z^{n} \,T_{*}\left( \indf{ \left\{ r=n \right\}}\,\eta \right)(x,y)
  }\\
  &=
  \sup_{(x,y)\in\Lambda}
  \abs{
    \sum_{n=1}^{\infty} z^{n}\, [\indf{ \left\{ r=n \right\}}\circ T^{-1}](x,y)\, [\eta\circ T^{-1}](x,y)
  }\\
  &=
  \sup_{(x,y)\in\Lambda}
  \abs{
    z^{r\left( T^{-1}(x,y) \right)}\, [\eta\circ T^{-1}](x,y)
  }\\
  &\le
  \sup_{(x,y)\in\Lambda}
  \abs{
    [\eta\circ T^{-1}](x,y)
  }\\
  &\le
  \Nsup{\eta}
\end{align*}
For the $\MUL{\cdot}$-term, fix $\gamma \in \Gamma$ and $(x,y), (w,z) \in \gamma$. Note that $T^{-1}\gamma$ is a segment of some unstable curve $\gamma' \in \Gamma$. Note that $r\circ T^{-1}$ is constant on unstable curves, that is $T_{*}r$ is $\mathcal{B}^{u}$ measurable. Let $N(\gamma) = r\left( T^{-1}(x,y)\right) = r\left( T^{-1}(w,z)\right)$. Let $x'$ and $w'$ denote the first coordinates of $T^{-1}(x,y)$ and $T^{-1}(w,z)$ respectively. Note that $\abs{x'-w'} \le \beta \abs{x-w}$, since $T$ is uniformly expanding along unstable curves. Computing as above we obtain,
\begin{align*}
  R(z)\eta(x,y) - R(z)\eta(w,z)
  &=
  z^{N(\gamma)}\eta\left( T^{-1}(x,y) \right)
  -
  z^{N(\gamma)}\eta\left( T^{-1}(w,z) \right)\\
  &\le
  \abs{z^{N(\gamma)}}\MUL{\eta}\abs{x'-w'}\\
  &\le
  \beta\MUL{\eta}\abs{x-w}.
\end{align*}	
Since $(x,y)$ and $(w,z)$ were arbitrary, $\MUL{R(z)\eta} \le \beta \MUL{\eta}$.
We conclude that,
\[
  \NUL{R(z)\eta} \le \Nsup{\eta} + \beta \MUL{\eta} \le \NUL{\eta}.
\]
\end{proof}
\begin{lem}
\label{lem:Peripheral Spectrum}
For each $z$ with $\abs{z}=1$,
\begin{enumerate}
  \item The peripheral spectrum of $R(z)$ consists of semi-simple\footnote{An eigenvalue is semi-simple if its algebraic and geometric multiplicities match.} eigenvalues.
  \item Every peripheral eigenvector of $R(z)$ is in $\ULip$.
\end{enumerate}
\end{lem}

\begin{proof}
This follows from \Cref{lem:Auxiliary Operator Bound} by a standard argument, and can be found in a slightly different setting in \cite{Baladi2000} Proposition 3.5. We will outline the proof for the convenience of the reader.\par

Note that by \Cref{prop:Quasi-Compactness} the operator $R(z)$ is quasi-compact and therefore admits a decomposition $R(z) = Q + F$ such that the spectral radius of $Q$ is the essential spectral radius of $R(z)$, $F$ is supported on a finite dimensional subspace of $\Bs$, and $QF = FQ = 0$. It follows directly that $R(z)^k = Q^k +F^k$ for $k\ge 1$.\par

It follows from \Cref{eqn:Lasota-Yorke Strong-Strong Inequality} that for all $k \ge 1$ and $\eta \in \ULip$, we have the uniform bound $\NBs{R(z)^{k}\eta} \le \abs{z}\left[ \kappa +1 \right]\NBs{\eta}$. From the decomposition in the last paragraph we have $\NBs{F^k\eta} \le \abs{z}\left[ \kappa +1 \right]\NBs{\eta}$.\par

The peripheral spectrum of $R(z)$ coincides with the peripheral spectrum of $F$ since the spectral radius of $Q$ is strictly less that that of $R(z)$. It follows that the peripheral spectrum of $R(z)$ consists of eigenvalues.\par

Recall that the spectral radius of $R(z)$ and hence $F$ is $\abs{z}$. If $F$ had a generalized eigenvector associated to a peripheral eigenvalue, then $\abs{z}^{-k}\|F^{k}\|_{op}$ would grow linearly in $k$. This cannot be the case since $\abs{z}^{-k}\|F^{k}\|_{op} \le \kappa+1$. Therefore, the peripheral spectrum of $R(z)$ consists of semi-simple eigenvalues.\par

Note that $F(\Bs)$ is finite dimensional and therefore closed in $\Bs$. By definition $\ULip$ is dense in $\Bs$, thus $F(\ULip)$ is dense in $F(\Bs)$. Since $F(\Bs)$ is closed $F(\ULip) = F(\Bs)$.\par

By \Cref{lem:Auxiliary Operator Bound}, $R(z)$ is bounded on $\ULip$ and thus $F$ is also. Therefore, $F(\Bs)=F(\ULip) \subset \ULip$. Every eigenvector of $F$ is contained in $F(\Bs)$. We conclude that every eigenvector associated to a peripheral eigenvalue of $R(z)$ is in $\ULip$.  
\end{proof}

We are now in a position to prove \Cref{prop:Spectrial Gap and Aperiodicity}.

\begin{proof}[Proof of \Cref{prop:Spectrial Gap and Aperiodicity}]
The proof of this lemma will be divided into several distinct parts.
\paragraph{Claim 1:} For all $\abs{z}\le 1$ the operator $R(z) - I$ is invertible if and only if $1$ is not an eigenvalue of $R(z)$.

\begin{subproof}[Proof of Claim 1]
  If $1$ is an eigenvalue of $R(z)$, then $R(z)-I$ is not invertible by the definition of an eigenvalue. Suppose that $R(z)-I$ is not invertible. Then $1$ is a point in the spectrum of $R(z)$. By \Cref{prop:Quasi-Compactness} the operator $R(z)$ is quasi-compact with essential spectral radius less than $\beta^{a}\abs{z}$, which is strictly less then $1$, therefore $1$ is a point in the spectrum of $R(z)$ that is outside the essential spectrum. It follows that $1$ is an eigenvalue of $R(z)$ and that any eigenvector associated to the eigenvalue $1$ lies in a finite dimensional $R(z)$ invariant subspace of $\Bs$.\par
\end{subproof}

\paragraph{Claim 2:} If $\abs{z}<1$, then $R(z)-I$ is invertible.\par

\begin{subproof}[Proof of Claim 2]
  Fix $z$ such that $\abs{z}<1$. It follows from \Cref{prop:Quasi-Compactness} that the spectral radius of $R(z)$ is at most $\abs{z}$. By assumption $\abs{z}<1$, so $1$ is not an eigenvalue of $R(z)$. By the previous claim $R(z)-I$ is invertible.\par
\end{subproof}

\paragraph{Claim 3:} If $\abs{z}=1$ and $z\neq 1$, then $I-R(z)$ is invertible. The operator $R(1)$ has a simple eigenvalue at $1$ and the associated eigenspace is $span\left\{ \indf{\Lambda} \right\}$.\par

\begin{subproof}[Proof of Claim 3]
  We will verify both parts of the claim simultaneously. Let $z$ be a complex number such that $\abs{z}=1$ and let $\eta\in\Bs$ be an eigenvector of $R(z)$ with eigenvalue $1$, that is
  \[
    R(z)\eta = \eta.
  \]
  The proof relies on two observations about $\eta$: 
\paragraph{Observation 1:} If $\abs{z} = 1$ and $\eta \in \Bs$ such that $R(z)\eta = \eta$, then for almost every $(x,y)\in\Lambda$,
\begin{equation}
  \label{eqn:Peripheral Spectrum Koopman Identity}
  [\eta\circ T](x,y)\, z^{r} = \eta(x,y).
\end{equation}
\paragraph{Observation 2:} If $\abs{z} = 1$ and $\eta\in \Bs$ so that $R(z)\eta = \eta$, then $\eta$ is a constant multiple of $\indf{\Lambda}$.\par

  We will verify both observations after completing the proof of Claim 3.\par

  We will show that, if $\eta \neq 0$, then $z=1$. By Observation 2, $\eta$ is constant. Since $T$ preserves Lebesgue measure $\eta\circ T = \eta$. It follows that \Cref{eqn:Peripheral Spectrum Koopman Identity} reduces to
  \[
    (z^{r(x)}-1)\eta = 0.
  \]
  The equation above is satisfied if $\eta = 0$ or if $z^{r(x)} =1$.\par

  The equation $z^{r(x)} = 1$ is satisfied if and only if for all $a \in \image{r} \subseteq \Z$,
  \[
    a \frac{\arg(z)}{2\pi} \in \Z. 
  \]
  The inclusion above can hold if and only if there exists a rational number $b/c$ such that $\frac{\arg(z)}{2\pi} = b/c$.
  Assuming that $b/c$ is reduced we see that $ab/c \in \Z$ and if and only if $c$ divides $a$.
  Therefore, $\frac{\arg(z)}{2\pi} = b/c$ and $c$ divides $a$ for all $a \in \image{r}$.
  From \Cref{sec:Escape From Fixed Points} it follows that $\image{r} = \left\{ n\in \N: n\ge2 \right\}$ and hence the greatest common divisor of $\image{r}$ is $1$ so that $c =1$ and hence $\frac{\arg(z)}{2\pi}\in\Z$.
  Therefore the principal value of the argument of $z$ is $0$ and hence $z=1$.\par

  $T$ preserves Lebesgue measure on $\Lambda$. By \Cref{eqn:Trasfer Operator at One} we have that $R(1)$ is the Frobenius-Perron operator of $T$. It follows that $R(1)\indf{\Lambda} =\indf{\Lambda}$. By Observation 2 any $\eta$ that satisfies the eigenvector equation $R(1)\eta =\eta$ is a multiple of $\indf{\Lambda}$. By \Cref{lem:Peripheral Spectrum} the eigenvalue $1$ is semi-simple. We have verified that $\indf{\Lambda}$ is a basis for the eigenspace associated to the eigenvalue $1$.  We conclude that $1$ is a simple eigenvalue of $R(1)$.

\end{subproof}

To complete the proof of the lemma it remains to verify Observation 1 and Observation 2 from the proof of the last claim.

\begin{subproof}[Proof of Observation 1]
  By \Cref{lem:Peripheral Spectrum} we have $\eta \in \ULip$.
  Since $\Ninfty{\eta} \le \Nsup{\eta} \le \NUL{\eta}$ we have $\eta \in \Linfty{\Lambda,\lambda}$.
  For all $\psi$ and  $\eta$ in $\ULip$ we have  
  \begin{align*}
    \int R(z) \eta \,\psi \, d\lambda 
    &=
    \int 
    \sum_{n=1}^{\infty}
    z^{n} R_{n}\eta\, \psi
    \,d\lambda
    =
    \sum_{n=1}^{\infty}
    \int 
    z^{n}T_{*}(\eta\,\indf{[r=n]}) \,\psi
    \,d\lambda\\
    &=
    \sum_{n=1}^{\infty}
    \int 
    \eta \,z^{n}\indf{[r=n]} \,\psi\circ T
    \,d\lambda
    =
    \int 
    \sum_{n=1}^{\infty}
    \eta\, z^{n}\indf{[r=n]} \,\psi\circ T
    \,d\lambda\\
    &=
    \int \eta \,z^{r} \psi\circ T \, d\lambda.
  \end{align*}
  Since $\eta \in \Linfty{\lambda}$ we have $\eta \in \Ltwo{\lambda}$.
  Define $W(z)$ on $\Linfty{\lambda}$ by $W(z) \psi= z^{r} \, \psi\circ T$.
  Now we compute as in \cite{Gouezel2004},
  \begin{align*}
    \Ntwo{W(z)\eta -\eta}^2 
    &= 
    \Ntwo{W(z)\eta}^2-2 Re\langle W(z)\eta , \eta \rangle +\Ntwo{\eta}^2 \\
    &= 
    \Ntwo{W(z)\eta}^2-2 Re\langle \eta , R(z)\eta \rangle +\Ntwo{\eta}^2\\
    &=
    \Ntwo{W(z)\eta}^2-2 Re\langle \eta , \eta \rangle +\Ntwo{\eta}^2\\
    &=
    \Ntwo{W(z)\eta}^2-\Ntwo{\eta}^2,
  \end{align*}
  and note that 
  \[
    \Ntwo{W(z)\eta}^{2} = \int |\eta|^2\circ T \, d\lambda = \int |\eta|^{2} \, d\lambda = \Ntwo{\eta}^{2},
  \]
  from which we conclude that $W(z)\eta = [\eta\circ T] \, z^{r} = \eta$ except possibly on a $\lambda$ null set. We have verified \Cref{eqn:Peripheral Spectrum Koopman Identity}.
\end{subproof}

\begin{subproof}[Proof of Observation 2]
  We begin by showing that $\eta$ is essentially constant along stable fibres. Recall that by \Cref{lem:Peripheral Spectrum} we have $\eta \in \ULip$.
  For each $j\ge1$ select $\tau_{j}\in C^{\infty}\left( \Lambda \right)$ such that $\None{\tau_{j}-\eta} <2^{-j}$.
  Note that $\None{W(\tau_{j}-\eta)} = \None{z^{r} (\tau_{j}-\eta)\circ T} = \None{\tau_{j}-\eta} <2^{-j}$.
  Let $\bar{\tau_{j}}(x,y) = \int \tau_{j}(x,y)\, dy$ and note that by the Mean Value Theorem there exists $s\in(0,1)$ and $t \in (y,s)$ such that 
  \[
    |\tau_{j}(x,y) - \bar{\tau}_{j}(x,y)| 
    = 
    |\tau_{j}(x,y) - \tau_{j}(x,s)| 
    =
    |\partial_y \tau_{j}(x,t)||y-s|
    \le 
    \Ninfty{\partial_y \tau_{j}}|y-s|.
  \]
  Further application of the Mean Value Theorem yields
  \[
    |W^n\tau_{j}(x,y) - W^n\bar{\tau}_{j}(x,y)|
    \le
    \Ninfty{\partial_y \tau_{j}}\Ninfty{\partial_y v^{(n)}_{x}}
    \le
    \Ninfty{\partial_y \tau_{j}} \beta^n.
  \]
  For each $j \ge 1$ select $n=n(j)$ such that $\Ninfty{\partial_{y}\tau_{j}}\beta^{n} + 2^{-j}< 10 \cdot 2^{-j}$ and note that 
  \[
    \None{\eta-\bar{\tau}_{j}} \le \None{W^{n}\eta- W^{n}\tau_{j}} + \None{W^{n}\tau_{j} - W^{n}\bar{\tau}_{j}} \le 10\cdot 2^{-j}.
  \]
  We see that $\eta$ is the $L^{1}$-limit of functions that are constant along stable fibres.
  It follows that for $\mu$-a.e. $x\in[p,q]$,
  \begin{equation}
    \label{eqn:Peripheral Spectrum Essentially Veriticaly Constant}
    \text{for $\Leb$-a.e. $y$, }\eta(x,y) = \int_{\ell(x)} \eta(x,z) \,d\Leb(z),
  \end{equation}
  Next we will use the unstable regularity of $\eta$ to show that Property \ref{eqn:Peripheral Spectrum Essentially Veriticaly Constant} holds for every $x \in [p,q]$. To verify this suppose that $x$ failed to satisfy Property \ref{eqn:Peripheral Spectrum Essentially Veriticaly Constant}. This can happen if and only if there exist sets $A_{x}, B_{x} \subset \ell(x)$ and $\epsilon>0$, such that $\Leb(A_x)>0$, $\Leb(B_x)>0$, and for all $y$ in $A_x$ and $z$ in $B_{x}$
  \begin{equation}
    \label{eqn:Peripheral Spectrum Separation of Values}
    \eta(x,y) - \eta(x,z)\ge \epsilon.
  \end{equation}
  For $w\neq x$ let $A_{w}\subset \ell(w)$ be the set obtained by sliding\footnote{By sliding along unstable curves we mean $(x,y) \mapsto \gamma(x,y) \cap \ell(w)$} $A_{x}$ along unstable curves into $\ell(w)$ and let $B_{w}$ be defined similarly. Note that $\Leb(A_{w})>0$ if and only if $\Leb(A_{x})>0$. Since $\eta$ is in $\ULip$ we have that 
  \[
    \abs{\eta(x,y) - \eta(\ell(w) \cap \gamma(x,y))} \le \MUL{\eta}\abs{x-w}.
  \]
  Choose $\delta>0$ so that $\MUL{\eta}\delta<\epsilon/3$.
  Fix $w\in [p,q]$ such that $\abs{w-x}<\delta$.
  Select $(w,y)\in A_{w}$ and $(w,z)\in B_{w}$ and let $(x,y')\in A_{x}$ and $(x,z')\in B_{x}$ denote the points obtained by sliding along unstable disks back to $\ell(x)$. We compute,
  \[
    \eta(w,y) - \eta(w,z) \ge \eta(x,y') - \eta(x,z') - 2\MUL{\eta}\abs{x-w} \ge \epsilon - 2\MUL{\eta}\delta \ge \frac{\epsilon}{3}. 
  \]
  We have just shown that for every $w \in [p,q]$ with $\abs{w-x}<\delta$ Property \ref{eqn:Peripheral Spectrum Separation of Values} holds at $w$, thus Property \ref{eqn:Peripheral Spectrum Essentially Veriticaly Constant} fails at $w$. This contradicts our observation that \Cref{eqn:Peripheral Spectrum Essentially Veriticaly Constant} holds for $\mu$-a.e. $x \in [p,q]$.
  We conclude that \Cref{eqn:Peripheral Spectrum Essentially Veriticaly Constant} holds for every $x\in[p,q]$.\par

  Define $h(x) =\int_{0}^{1} \eta(x,y) \, dy$.
  This function is Lipschitz. To verify this fix $x,w \in [p,q]$. Let $A_{x}\subset \ell(x)$ denote the set of points in $\ell(x)$ where \Cref{eqn:Peripheral Spectrum Essentially Veriticaly Constant} fails and let $A_{w}$ be defined similarly. By the previous paragraph both $A_{x}$ and $A_{w}$ are null sets. Let $B \subset \ell(x)$ be the set obtained by sliding $A_{w}$ along unstable disks into $\ell(x)$. The set $B$ is a null set, therefore the set $G=\ell(x) -(A_x \cup B)$ consisting of points in $\ell(x)$ where $\eta(x,y) = h(x)$ and $\eta(\gamma(x,y) \cap \ell(w)) = h(w)$ has full measure. Choose $(x,y) \in G$ and note that
  \[
    \abs{h(x) - h(w)} = \abs{\eta(x,y) - \eta\left(\gamma(x,y)\cap \ell(x)  \right)} \le \MUL{\eta}\abs{x-w},
  \]
  so $h$ is Lipschitz with Lipschitz constant at most $\MUL{\eta}$.\par

  Next we would like to verify $\int [W(z)\eta](x,y)\,dy = z^{r}\, [h\circ u](x)$.
  Note that $T$ maps $\ell(x)$ into $\ell(u(x))$ affinely. We will apply the change of variable $y' = v_{x}(y)$ noting that $dy' = \partial_{y}v_{x}(y) dy$ and that $\partial_{y} v_{x}(y)$ is constant and exactly equal to the length of the interval $T\ell(x) \subset \ell(u(x))$
  \begin{align*}
    \int_{0}^{1} z^{r(x)} (\eta\circ T)(x,y) \, dy
    =
    z^{r(x)} \frac{1}{\abs{T\ell(x)}}\int_{T\ell(x)} \eta(u(x),y')\, dy'
    =
    z^{r(x)}h\left( u(x) \right)
  \end{align*}
  Applying Observation 1 we obtain
  \begin{equation}
    \label{eqn:Peripheral Spectrum Factor Koopman Identity}
    z^{r}\, [h\circ u](x) = h(x)
  \end{equation}

  Next we deduce that $h$ is an essentially constant function. We will apply Corollary 3.2 from \cite{Aaronson&Denker-2001}. We reformulate the Corollary in our notation for the convenience of the reader.\par

  {\it
    Suppose that:
    \begin{itemize}
      \item $u\colon [p,q]  \to  [p,q] $ is a probability preserving, almost onto Gibbs-Markov map with respect to the partition $\alpha = \left\{ I_{j},I_{j}':j=2,\cdots,\infty \right\}$\footnote{see \Cref{sec:Escape From Fixed Points}}.
      \item $\varphi \colon [p,q] \to \left\{ z\in\C:\abs{z}=1 \right\}$ is $\alpha$-measurable.
      \item $h \colon [p,q] \to \left\{ z\in\C:\abs{z}=1 \right\}$ is Borel measurable and $\varphi(x) = h\cdot \bar{h}\circ u$
    \end{itemize}
  Then $h$ is essentially constant.}\par

  Let us verify that $u$ satisfies the first hypothesis of the Corollary. For each $a \in \alpha$ the map $u|_{a}$ is a homeomorphism onto $[p,q)$ with $C^{2}$ inverse $v_{a}\colon [p,q] \to a$.
  The map $u$ is uniformly expanding by \Cref{lem:Uniform Expansion} and satisfies Adler's bounded distortion property by \Cref{lem:Distortion Bound}. By Example 2 of \cite{Aaronson&Denker-2001} it follows that $u$ is a mixing Gibbs-Markov map. Since every branch of $u$ is onto, $u$ is almost onto as defined immediately after Theorem 3.1 of \cite{Aaronson&Denker-2001}.\par

  Since $u$ is a Gibbs-Markov map, $u$ is ergodic. Taking the complex modulus of \Cref{eqn:Peripheral Spectrum Factor Koopman Identity} yields $\abs{h} = \abs{h\circ u} = \abs{h}\circ u$, thus $\abs{h}$ is an essentially constant function. Since $h$ is Lipschitz, we have that $\abs{h}$ is Lipschitz and therefore pointwise constant. Without loss of generality assume that $\abs{h} = 1$.\par

  Since $h$ is a circle valued function we have $\bar{h} = 1/h$. Let $\varphi(x) = h\cdot\bar{h}\circ u$. By \Cref{eqn:Peripheral Spectrum Factor Koopman Identity} we have 
  \[
    \varphi(x) = h\cdot\bar{h}\circ u = \frac{h}{h\circ u} = z^{r(x)}.
  \]
  Since $r(x)$ is measurable with respect to the partition $\alpha$ we have that $\varphi$ is circle valued and $\alpha$-measurable. We have just verified that $\varphi$ satisfies the second hypothesis above and that $h$ and $\varphi$ are related as required in the third hypothesis by definition.\par

  Applying the Corollary we see that $h$ is essentially constant. Since $h$ is Lipschitz we conclude that $h$ is pointwise constant. Let $h_{0}$ denote the constant value of $h$. \par

  Define $H(x,y) = h_{0}$, this function is clearly in $\ULip$. On each vertical line the function $H$ agrees with $\eta$ except possibly on a set of one dimensional Lebesgue measure zero. It follows that for all $t \in [p,q]$ there exists a $\lambda$-null set $N_{t}$ such that for all $(x,y) \in \Lambda - N_{t}$ we have $\eta\left( \ell(t)\cap \gamma(x,y) \right) -H(x,y) =0 $. With this fact it follows directly from \Cref{eqn:Stable Norm Defintion,eqn:Unstable Modulus Defintion} that $\Ns{\eta -H} =0$ and $\MBs{\eta -H} = 0$, thus $\NBs{\eta - H} = 0$. We conclude that $\eta$ and $H$ are in the same $\Bs$-equivalence class.
\end{subproof}
Having verified Observation 1 and Observation 2 from the proof of Claim 3 we see that the lemma follows by combining Claim 2 and Claim 3.
\end{proof}

\subsection{Proof of Finite Variance Conditions}
\label{sec:Finite Variance Conditions}
In this section we prove \Cref{lem:Base Observable is L2}.
\begin{proof}[Proof of \Cref{lem:Base Observable is L2}]
  Recall that $\Omega_{1}$ is a partition mod $\lambda$ of $\Lambda$ that is the common refinement of the return time partition and the partition that splits $\Lambda$ along the vertical line $\ell_{A}$ (see \Cref{eqn:Stable Columns Partition}. Let $\omega_{1}(n,0)$ to be the cell of $\Omega_{1}$ with return time $n$ that lies to the right of $\ell_{A}$ and $\omega_{1}(n,1)$ be the cell of $\Omega_{1}$ that lies to the left of $\ell_{A}$. Note that $\omega_{1}(n,j)$ must pass near the fixed line at $\ell_{j}$ before returning to $\Lambda$.
  By the monotone convergence theorem and H\"older's inequality
  \begin{align*}
    \int_{\Lambda} \abs{\xi}^{2} \, d\lambda 
    &= 
    \sum_{j=0}^{1}
    \sum_{n=0}^{\infty} 
    \int_{\Lambda} 
    \indf{\omega_{1}(n,j)} \abs{\xi}^2\, d \lambda\\
    &\le
    \sum_{j=0}^{1}
    \sum_{n=0}^{\infty} 
    \Nsup{\indf{\omega_{1}(n,j)} \abs{\xi}^2} \lambda\left( \omega_{1}(n,j) \right).
  \end{align*}
  By \Cref{eqn:Interval Orbit Structure} we have $\lambda\left( \omega_{1}(n,0)\right)= p_{n+1}^{\circ} - p_{n+2}^{\circ}$ and $\lambda\left( \omega_{1}(n,0) \right)= q_{n+2}^{\circ} - q_{n+1}^{\circ}$. By \Cref{eqn:Period-2 Orbit Increment Asymptotics Interior Right,eqn:Period-2 Orbit Increment Asymptotics Interior Left}, as $n \to \infty$,
  \[
    \lambda\left( \omega_{1}(n,j)\right) 
    = 
    O\left( \left(\tfrac{1}{n} \right)^{2+\frac{1}{\alpha_{j}}}\right).
  \]
  We will show that 
  \[
    \sum_{n=0}^{\infty} 
    \Nsup{\indf{\omega_{1}(n,0)} \abs{\xi}^2} \lambda\left( \omega_{1}(n,0) \right),
    \tag{$\star$}
  \]
  converges for each of the stated conditions. The sum with $j=1$ converges by analogous arguments that are independent of the $j=0$ case.

  \begin{enumerate}[i.]
    \item By \Cref{lem:Base Observable Expansion}, as $n\to \infty$
      Suppose that $\alpha_{0}<1$
      \[
        \Nsup{\indf{\omega_{1}(n,0)}\abs{\xi}^2} =O(n^{2}).
      \]
      Therefore,
      \[
        \Nsup{\indf{\omega_{1}(n,0)} \abs{\xi}^2} \lambda\left( \omega_{1}(n,0) \right)
        =
        O\left( \left(\tfrac{1}{n} \right)^{\frac{1}{\alpha_{0}}}\right).
      \]
      and the terms in ($\star$) are summable.

    \item Suppose that $M_{0}=0$ and $\alpha_{0}=1$.
      By \Cref{lem:Base Observable Expansion}, as $n\to \infty$,
      \[
        \Nsup{\indf{\omega_{1}(n,0)}\abs{\xi}^2} =O(n^{2-2\gamma}).
      \]
      Therefore,
      \[
        \Nsup{\indf{\omega_{1}(n,0)} \abs{\xi}^2} \lambda\left( \omega_{1}(n,0) \right)
        =
        O\left( \left(\tfrac{1}{n} \right)^{2\gamma + 1}\right).
      \]
      Since $2\gamma+1 > 1$, the terms in ($\star$) are summable.
    \item Suppose that $M_{0} = 0$, $1< \alpha_{0}<3$, and $\gamma>\tfrac{\alpha_{0}-1}{2}$.
      By \Cref{lem:Base Observable Expansion}, as $n\to \infty$,
      \[
        \Nsup{\indf{\omega_{1}(n,0)}\abs{\xi}^2} =O\left(n^{2-2\frac{\gamma}{\alpha_{0}}}\right).
      \]
      Therefore,
      \[
        \Nsup{\indf{\omega_{1}(n,0)} \abs{\xi}^2} \lambda\left( \omega_{1}(n,0) \right)
        =
        O\left( \left(\tfrac{1}{n}\right)^{\frac{2\gamma +1}{\alpha_{0}}} \right)
      \]
      Since $\tfrac{2\gamma +1}{\alpha_{0}} > 1$, the terms in ($\star$) are summable.

  \end{enumerate}
\end{proof}
\subsection{Proof of Convergence and Continuity of Perturbations}
  \label{sec:Convergence and Continuity of Purturbations}
  In this section we prove \Cref{prop:Convergence and Continuity of Perturbations}
  \begin{proof}[Proof of \Cref{prop:Convergence and Continuity of Perturbations}]
  We will show that for all $\epsilon>0$ there exists $\delta>0$ such that for all $t \in (-\delta,\delta)$ and for all $\eta \in \ULip$,
  \begin{equation*}
    \NBs{\left[R_{n}(t) - R_{n}(0)\right]\eta} \le \epsilon \NBs{\eta}.
  \end{equation*}
  First note that,
  \begin{equation*}
    \left[ R_{n}(t)-R_{n}(0) \right]\eta = T_{*}\left( \left( \exp\left(it\xi\right) -1\right) \indf{[r=n]} \eta\right)
  \end{equation*}
  Second note that,
  \begin{equation*}
    \NBs{\left[R_{n}(t) - R_{n}(0)\right]\eta}
    =
    \Ns{\left[R_{n}(t) - R_{n}(0)\right]\eta}
    +
    \MBs{\left[R_{n}(t) - R_{n}(0)\right]\eta}.
  \end{equation*}
  Fix $\epsilon >0$. We will estimate the two terms on the right. Let $\psi \in \Ta$ be a test function with $\NTa{\psi}\le1$, $x$ be a point in $[p,q]$, and consider a typical integral from the definition of
  $\Ns{\left[R_{n}(t) - R_{n}(0)\right]\eta}$,
  \begin{equation*} 
    I=
    \int_{0}^{1}
    T_{*}\left(\left( \exp\left(it\xi  \right)-1 \right)\indf{[r=n]}\eta \right)(x,y)
    \, \psi(y)\, dy
  \end{equation*} 
  We will use the following facts to bound $I$.
  \begin{enumerate}[A.]
    \item For all $x \in \left[ p,q \right]$, $\frac{dv_{x}}{dy}$ is a constant and for all $x \in [r = n]$, $0< \frac{dv_{x}}{dy} \le \left[ \kappa +1 \right]\lambda\left[ r=n \right]$.
    \item The value of $\indf{\left[ r=n \right]}(x,y)$ is independent of $y$.
    \item For all $t,a \in \R$, $\abs{\exp\left( ita \right)-1} \le 2\abs{\sin\left(\frac{at}{2}\right)} \le \abs{at}$. 
      %Therefore, for all $(x,y), (x,z) \in \Lambda$,
      %\begin{align*}
      %  \abs{\exp\left( it \xi(x,y) \right)-1} 
      %  &\le 
      %  \abs{t} \abs{\xi(x,y)}\\
      %  \frac{
      %    \abs{
      %    \exp\left( it \xi(x,y) \right) 
      %    -
      %    \exp\left( it \xi(x,z) \right) 
      %    }
      %  }{\abs{y-z}}
      %  &\le
      %  \abs{
      %    \exp\left( it \xi(x,z) \right)
      %  }
      %  \frac{
      %    \abs{\exp\left( it\left( \xi(x,y) - \xi(x,z) \right) \right)-1}
      %  }{
      %    \abs{y-z}
      %  }\\
      %  &\le 
      %  \abs{t}
      %  \frac{
      %    \abs{
      %    \left( 
      %      \xi(x,y)
      %      - 
      %      \xi(x,z) 
      %    \right) 
      %    }
      %  }{
      %    \abs{y-z}
      %  }\\
      %  &\le 
      %  \abs{t} 
      %  H^{1}\left( \xi\left( x,\cdot \right) \right).
      %\end{align*}
      It follows that, for all $x \in [p,q]$, 
      \[
        \NTl{\exp\left( it\xi(x,\cdot) \right)-1} \le \abs{t} \NTl{\xi\left( x,\cdot \right)}.
      \]
  \end{enumerate}
  Applying the change of variables $(x,y) = T(w,z) = (u(x),v_{w}(z))$, we compute as follows.
  \begin{align*}
    I
    &= 
    \int_{0}^{1}
    T_{*}\left(\left( \exp\left(it\xi  \right)-1 \right)\indf{[r=n]}\eta \right)(x,y)
    \, \psi(y)\, dy\\
    &= 
    \int_{0}^{1}
    \left(\left( \exp\left(it\xi  \right)-1 \right)\indf{[r=n]}\eta \right)(w,z)
    \, \psi(v_{w}(z))\, \frac{dv_{w}}{dz} \, dz \\
    &\le 
    [\kappa +1] \lambda[r=n] 
    \int_{0}^{1}
    \left(\left( \exp\left(it\xi  \right)-1 \right)\indf{[r=n]}\eta \right)(w,z)
    \, \psi(v_{w}(z))\, dz\tag{by A}\\
    &= 
    [\kappa +1] \lambda[r=n] 
    \int_{0}^{1}
    \eta(w,z)
    \left( \exp\left(it\xi  \right)-1 \right)(w,z)
    \, \psi(v_{w}(z))\, dz \tag{by B}\\
    &\le
    [\kappa +1] \lambda[r=n] 
    \NBs{\eta}
    \NTl{
      \psi
      \left( \exp\left(it\xi  \right)-1 \right)(w,\cdot)
    }\\
    &\le
    2
    [\kappa +1] \lambda[r=n] 
    \NBs{\eta}
    \NTl{
      \psi
    }
    \NTl{
      \left( \exp\left(it\xi  \right)-1 \right)(w,\cdot)
    }\\
    &\le
    2\abs{t}
    [\kappa +1] \lambda[r=n] 
    \NBs{\eta}
    \NTl{
      \psi
    }
    \NTl{
      \xi(w,\cdot)
    }\tag{by C}
  \end{align*}
  Similarly, we consider a typical integral from the definition of $\MBs{\left(R_{n}(t) - R_{n}(0)\right)[\eta]}$. Let $x_1, x_2\in [p,q]$ such that $x_1 \neq x_2$ and $\psi_{1}$ and $\psi_{2}$ be test functions. We will apply changes of variable $(x_{1},y) = T(w_{1},z_{1})$ and $(x_{2},y) = T\left( w_{2},z_{2} \right)$. Define $\Psi_{j} = \psi_{j}\circ v_{w_{j}} \left( \exp(it\xi)-1 \right)$ and $\Phi_{j} = \left(\max_{j}\left\{\NTl{\Psi_j}\right\}\right)^{-1}\Psi_{j}$. We compute as follows.
  \begin{align*}
    II
    &= 
    \int_{0}^{1} \frac{R_{n}(t)\eta(x_1,y)\psi_{1}(y)-R_{n}(t)\eta(x_2,y)\psi_{2}(y)}{\abs{x_1-x_2}} \, dy\\
    &= 
    \tfrac{1}{\abs{x_{1}-x_{2}}} \int_{0}^{1} \eta(w_{1},z_{1})\Psi_{1}(z_{1})\, \frac{dv_{w_{1}}}{dz_{1}}\, dz_{1}\\
    &-
    \tfrac{1}{\abs{x_{1}-x_{2}}} \int_{0}^{1} \eta(w_{2},z_{2})\Psi_{2}(z_{2})\, \frac{dv_{w_{2}}}{dz_{2}}\, dz_{2}\\
    &\le
    \tfrac{[\kappa+1]\lambda[r=n]}{\abs{x_{1}-x_{2}}} \int_{0}^{1} \eta(w_{1},z_{1})\Psi_{1}(z_{1})\, dz_{1}\\
    &-
    \tfrac{[\kappa+1]\lambda[r=n]}{\abs{x_{1}-x_{2}}} \int_{0}^{1} \eta(w_{2},z_{2})\Psi_{2}(z_{2})\, dz_{2}\\
    &\le
    \tfrac{[\kappa+1]\lambda[r=n]\max_{j}\left\{\NTl{\Psi_j}\right\}}{\abs{x_{1}-x_{2}}} \int_{0}^{1} \eta(w_{1},z)\Phi_{1}(z) - \eta(w_{2},z)\Phi_{2}(z)\, dz\\
    &\le
    [\kappa+1]\lambda[r=n]\max_{j}\left\{\NTl{\Psi_j}\right\}\MBs{\eta}\\
    &\le
    2\abs{t}[\kappa+1]\lambda[r=n]\max_{j}\left\{\NTl{\psi_j}\NTl{\xi(w_{j},\cdot)}\right\}\MBs{\eta}
  \end{align*}
  It is necessarily the case that $w,w_{1},w_{2}\in[r=n]$ in the calculations above. 
  By \Cref{lem:Base Observable Expansion}, $\NTl{\xi(x,\cdot)}$ is uniformly bounded on $[r=n]$ by some constant $M(n)$. Chose $\delta>0$ such that $\delta(\kappa +1)M(n)\Ns{\eta}<\epsilon/3$. It follows from the calculations above that for $\abs{t}<\delta$ we have $I<\epsilon/3$ and $II<2\epsilon/3$, which completes the proof of continuity.\\

  By the estimates above, 
  \[
    \NBs{R_n(t) - R_{n}(0)} = O\left( \abs{t}\lambda\left[ r=n \right] \|\xi\|_{v_{1}}^{1}\right).
  \]

  As $n\to \infty$, $\NTl{\xi(x,\cdot)} = O(n)$ and $\lambda\left[ r=n \right] = O\left(n^{-2-\frac{1}{\alpha}}\right)$, where $\alpha = \max\left\{ \alpha_{0},\alpha_{1} \right\}$, thus
  \[
    \|R_n(t) - R_{n}(0)\|_{op} = O\left( \abs{t}n^{-1-\frac{1}{\alpha}}\right).
  \]
  The estimate above is summable in $n$. The result for $R(z,t)$ follows by an easy application of the triangle inequality and monotone convergence.\par

  By unpacking the definition of $\NBs{R_{n}(t)\eta}$ as was done above and choosing to use the bound$\abs{e^{-it\xi}}\le1$ yields
    \[
      \|R_n(t)\|_{op} = O\left( \lambda\left[ r=n \right] \right).
    \]
    Note that we could have obtained a similar bound in the argument above by taking $2\abs{\sin\left( \frac{t\xi}{2} \right)}<2$ instead of $2\abs{\sin\left( \frac{t\xi}{2} \right)}<\abs{t}\abs{\xi}$, however we would not have obtained the desired dependence on $\abs{t}$.
    Since $\lambda\left( [r=n] \right) = O \left( n^{-2-\frac{1}{\alpha}} \right)$, where $\alpha = \max\left\{\alpha_{0},\alpha_{1}\right\}$, the sequence $r_{n} = \lambda\left[ r=n \right]$ verifies the third claim.

\end{proof}
\subsection{Proof of Expansion of the Dominant Eigenvalue}
\label{sec:Expansion of Dominant Eigenvalue}
In this section we will prove \Cref{prop:Expansion of Dominant Eigenvalue}. Before we can complete the proof we will need the following two lemmas on the asymptotic behavior of $\xi$.
\begin{lem}
  \label{lem:Base Observable Expansion}
  Suppose that $X\colon \left[ 0,1 \right]^{2} \to \R$ is $\gamma$-H\"older for some $\gamma\in(0,1]$ and that $(x,y) \in \Lambda$ is a point such that $x\in [A,q]$ and $r(x,y) = n+2$ for some $n\ge0$. As $n \to \infty$,
  \[
    \xi(x,y) =  n\int_{0}^{1} X(0,y^{1+\frac{1}{\alpha_{0}}})\,dy + O\left(n^{1-\gamma}\right) + O\left(n^{1-\frac{\gamma}{\alpha_{0}}}\right).
  \]
  If $x\in [p,A]$, then as $n \to \infty$,
  \[
    \xi(x,y) =  n\int_{0}^{1} X(1,y^{1+\frac{1}{\alpha_{1}}})\,dy + O\left(n^{1-\gamma}\right) + O\left(n^{1-\frac{\gamma}{\alpha_{1}}}\right).
  \]
\end{lem}
\begin{proof}
  We will prove the first asymptotic expansion, the proof of the second is similar. Through out this proof we will suppress the subscript on the contact parameters ($\alpha = \alpha_{0}$ and $c = c_{0}$). By \Cref{eqn:Base Observable Definition} 
  \[
    \xi(x,y)
    =
    \sum_{k=0}^{n+1}
    X\left(x_{k},y_{k}\right).
  \]
  Since $X$ is $\gamma$-H\"older,
  \begin{align*}
    \abs{
      X\left(x_{k},y_{k}\right)
      -
      X\left(0,y_{k}\right)
    }
    =
    O\left( x_{k}^{\gamma} \right)
    &= 
    O\left(n^{-\tfrac{\gamma}{\alpha}}\right).\\
    \abs{
      X\left(0,y_{k}\right)
      -
      X\left(0,\left(1- \tfrac{k+1}{n} \right)^{1+\frac{1}{\alpha}}\right)
    }
    &=
    O\left( n^{-\gamma} \right).
  \end{align*}
  An end point approximation to the Riemann sum shows that 
  \begin{align*}
    \abs{
      \int_{0}^{1}
      X\left(0,y^{1+\frac{1}{\alpha}}\right)
      \, dy
      -
      \frac{1}{n}
      \sum_{k=1}^{n-1}
      X\left(0,\left(1- \tfrac{k+1}{n} \right)^{1+\frac{1}{\alpha}}\right)
    }
    =O\left(n^{-\gamma}\right)
  \end{align*}
  A standard triangle inequality argument shows that 
  \[
    \xi(x,y) =  n\int_{0}^{1} X(0,y^{1+\frac{1}{\alpha}})\,dy + O\left(n^{1-\gamma}\right) + O\left(n^{1-\frac{\gamma}{\alpha}}\right)
  \]
  and therefore the claimed asymptotic holds.
\end{proof}
Next we investigate the cumulative distribution function of $\xi$.
\begin{lem}
  \label{lem:Base Observable Tail Asymptotics}
  Suppose that $X\colon \left[ 0,1 \right]^{2} \to \R$ is $\gamma$-H\"older for some $\gamma\in(0,1]$.
  \begin{itemize}
    \item If $M_0>0$, then for $t$ sufficiently large,
      \begin{align*}
        \lambda \left(\left[\xi >t \right] \cap [A,q]\right)
        &\sim 
        \tfrac{M_{0}}{\alpha_{0}\Leb(\Lambda)}
        \left( \tfrac{M_{0}\left( \alpha_{0}+1 \right)}{c_{0}\alpha_{0}} \right)^{\frac{1}{\alpha_{0}}}
        \left( \tfrac{1}{t} \right)^{1+\frac{1}{\alpha_{0}}},\\
        \lambda \left(\left[\xi <-t \right] \cap [A,q]\right)
        &=0.
      \end{align*}
    \item If $M_{0}<0$, then for $t$ sufficiently large,
      \begin{align*}
        \lambda \left(\left[\xi >t \right] \cap [A,q]\right)
        &=0,\\
        \lambda \left(\left[\xi <-t \right] \cap [A,q]\right)
        &\sim 
        \tfrac{\abs{M_{0}}}{\alpha_{0}\Leb(\Lambda)}
        \left( \tfrac{\abs{M_{0}}\left( \alpha_{0}+1 \right)}{c_{0}\alpha_{0}} \right)^{\frac{1}{\alpha_{0}}}
        \left( \tfrac{1}{t} \right)^{1+\frac{1}{\alpha_{0}}}.
      \end{align*}
    \item If $M_1>0$, then for $t$ sufficiently large,
      \begin{align*}
        \lambda \left(\left[\xi >t \right] \cap [p,A]\right)
        &\sim 
        \tfrac{M_{1}}{\alpha_{1}\Leb(\Lambda)}
        \left( \tfrac{M_{1}\left( \alpha_{1}+1 \right)}{c_{1}\alpha_{1}} \right)^{\frac{1}{\alpha_{1}}}
        \left( \tfrac{1}{t} \right)^{1+\frac{1}{\alpha_{1}}},\\
        \lambda \left(\left[\xi <-t \right] \cap [p,A]\right)
        &=0.
      \end{align*}
    \item If $M_{1}<0$, then for $t$ sufficiently large,
      \begin{align*}
        \lambda \left(\left[\xi >t \right] \cap [p,A]\right)
        &=0,\\
        \lambda \left(\left[\xi <-t \right] \cap [p,A]\right)
        &\sim 
        \tfrac{\abs{M_{1}}}{\alpha_{1}\Leb(\Lambda)}
        \left( \tfrac{\abs{M_{1}}\left( \alpha_{1}+1 \right)}{c_{1}\alpha_{1}} \right)^{\frac{1}{\alpha_{1}}}
        \left( \tfrac{1}{t} \right)^{1+\frac{1}{\alpha_{1}}}.
      \end{align*}
  \end{itemize}

\end{lem}
\begin{proof}
  We will prove the first asymptotic, the proofs of the others are similar. We will suppress subscripts through out this proof ($M = M_{0}$, $\alpha=\alpha_{0}$, and $c = c_{0}$). 
  For convenience define for any function $f$ on $\Lambda$ and real number $t$, $U(f,t) = \left[ f>t \right]\cap [A,q]$.
  Note that by \Cref{eqn:Interval Orbit Structure}, $\lambda\left(U(r,n)\right) =  \frac{p_{n}^{\circ} - A}{q-p} = \frac{p_{n}^{\circ} - A}{\Leb(\Lambda)}$, thus by \Cref{eqn:Period-2 Orbit Preimage Asymptotics Interior Right}
  \[
    \lambda\left(U(r,t)\right)
    \sim \tfrac{1}{\alpha\Leb(\Lambda)}
    \left( \tfrac{\left( \alpha+1 \right)}{c\alpha} \right)^{\frac{1}{\alpha}}
    \left( \tfrac{t}{\floor{t}} \right)^{1+\frac{1}{\alpha}}
    \left( \tfrac{1}{t} \right)^{1+\frac{1}{\alpha}}.
  \]
  Let $g(x,y) = \xi(x,y) - M r(x,y)$, then
  Fix $\epsilon>0$ and note that,
  \begin{align*}
    \lambda\left( U(\xi,t) \right)
    &\ge
    \lambda\left( U(Mr,t(1+\epsilon)) \right)
    -
    \lambda\left( U(\abs{g},\epsilon t) \right),\\
    \lambda\left( U(\xi,t) \right)
    &\le
    \lambda\left( U(Mr,t(1-\epsilon)) \right)
    +
    \lambda\left( U(\abs{g},\epsilon t) \right).
  \end{align*}
  Note that $\abs{g} > \epsilon t$ iff $r > \tfrac{r}{\abs{g}}\epsilon t$. By \Cref{lem:Base Observable Expansion} $\abs{g} = o\left( r(x,y) \right)$, thus the quantity $\tfrac{r}{\abs{g}} \epsilon$ is unbounded as $r \to \infty$. We conclude that as $t\to \infty$,
  \begin{equation*}
    \lambda\left( U(\abs{g},\epsilon t) \right) = o\left( \lambda\left( U(r,t) \right) \right).
  \end{equation*}
  Therefore as $t\to\infty$,
  \[
    \abs{
      \lambda\left( U(\xi,t) \right)
      -
      \lambda\left( U\left(r,\tfrac{t}{M}\right) \right)
    }
    =
    o\left(
    \lambda\left[ r>t \right]
    \right).
  \]
  The claimed asymptotic for $\lambda\left( U(\xi,t) \right)$ follows, since 
  \[
    \left(\frac{t}{M} \floor{\frac{M}{t}}\right)^{1+\frac{1}{\alpha}} = 1+o(1)
  \]
  as $t \to \infty$.\par

  It is not hard to check that $\xi$ is continuous on each set $\left[ r=n+2 \right] \cap [A,q]$ for $n\ge 0$. By \Cref{lem:Base Observable Expansion}, for $(x,y) \in \left[ r=n+2 \right] \cap [A,q]$, 
  \[
    \xi(x,y) = Mn + O\left( n^{1-\gamma} \right) + O\left( n^{1-\frac{\gamma}{\alpha}} \right).
  \]
  For $n$ sufficiently large the first term dominates the last two and $\xi$ is strictly positive on $\left[ r=n+2 \right] \cap [A,q]$. This leaves finitely many sets where $\xi$ may be negative, on each $\xi$ is continuous, therefore $\xi$ is bounded below. We conclude that, for $t$ sufficiently large,
  \[
    \lambda \left(\left[\xi <-t \right] \cap [A,q]\right)
    =0.
  \]

\end{proof}

We are now in a position to prove \Cref{prop:Expansion of Dominant Eigenvalue}.
\begin{proof}[Proof of \Cref{prop:Expansion of Dominant Eigenvalue}]
  By \Cref{prop:Convergence and Continuity of Perturbations} we have
  \[
    \NBs{R(z,t) - R(z,0)} = O\left( \abs{t} \right).
  \]
  If $e(t)$ is the eigenfunction of $R(1,t)$ associated to the eigenvalue $\chi(t)$ with integral $1$, then because eigenvectors depend holomorphicaly on operators
  \[
    \NBs{e(t)-1} 
    = 
    O\left( 
    \NBs{R(z,t) - R(z,0)} 
    \right)
    =
    O\left( \abs{t} \right).
  \]
  With this estimate in place the claimed expansions follow directly from \Cref{lem:Base Observable Tail Asymptotics} and the following theorems.
  \begin{enumerate}[i.]
    \item By arguments similar to \cite{Gouezel2004-Intermittent} Theorem 3.7 we obtain the claimed expansion. Since $T_{*}$ has a spectral gap the series in the definition of $\sigma^2$ converges.  
    \item The estimate above is sufficient to apply \cite{Aaronson&Denker-2001} Theorem 5.1, which yields the desired expansion of the eigenvalue $\chi(t)$ for $t$ near $0$.
    \item The estimate above is sufficient to apply \cite{Aaronson&Denker-2001} Theorem 5.1, which yields the desired expansion of the eigenvalue $\chi(t)$ for $t$ near $0$.
    \item Similarly we apply \cite{Aaronson&Denker-2001-normal} Theorem 3.1 to obtain the claimed expansion.
  \end{enumerate}
\end{proof}

\bibliographystyle{plain}
\bibliography{main_bibliography}

\begin{thebibliography}{10}

\bibitem{Aaronson&Denker-2001-normal}
Jon Aaronson and Manfred Denker.
\newblock {\em A Local Limit Theorem for Stationary Processes in the Domain of
  Attraction of a Normal Distribution}, pages 215--223.
\newblock Birkh{\"a}user Boston, Boston, MA, 2001.

\bibitem{Aaronson&Denker-2001}
Jon Aaronson and Manfred Denker.
\newblock Local limit theorems for partial sums of stationary sequences
  generated by gibbs--markov maps.
\newblock {\em Stochastics and Dynamics}, 1(02):193--237, 2001.

\bibitem{AlexanderYorke1984}
J.~C. Alexander and J.~A. Yorke.
\newblock Fat baker's transformations.
\newblock {\em Ergodic Theory and Dynamical Systems}, 4:1--23, 3 1984.

\bibitem{Baladi2000}
Viviane Baladi.
\newblock {\em Positive transfer operators and decay of correlations},
  volume~16.
\newblock World scientific, 2000.

\bibitem{Baladi&Tsujii2007}
Viviane Baladi and Masato Tsujii.
\newblock Anisotropic h{\"o}lder and sobolev spaces for hyperbolic
  diffeomorphisms.
\newblock In {\em Annales de l'institut Fourier}, volume~57, pages 127--154,
  2007.

\bibitem{Blank&Keller&Liverani2002}
Michael Blank, Gerhard Keller, and Carlangelo Liverani.
\newblock Ruelle--perron--frobenius spectrum for anosov maps.
\newblock {\em Nonlinearity}, 15(6):1905, 2002.

\bibitem{Bose&Murray2013}
Christopher Bose and Rua Murray.
\newblock Polynomial decay of correlations in the generalized baker's
  transformation.
\newblock {\em International Journal of Bifurcation and Chaos}, 23(08):1350130,
  2013.

\bibitem{Bose1989}
Christopher~J. Bose.
\newblock Generalized baker's transformations.
\newblock {\em Ergodic Theory and Dynamical Systems}, 9:1--17, 3 1989.

\bibitem{Chart2016}
Seth Chart.
\newblock {\em Polynomial decay of correlations for generalized baker’s
  transformations via anisotropic Banach spaces methods and operator renewal
  theory}.
\newblock PhD thesis, University of Victoria, 2016.
\newblock \url{http://hdl.handle.net/1828/7242}.

\bibitem{Demers&Liverani2008}
Mark Demers and Carlangelo Liverani.
\newblock Stability of statistical properties in two-dimensional piecewise
  hyperbolic maps.
\newblock {\em Transactions of the American Mathematical Society},
  360(9):4777--4814, 2008.

\bibitem{Gouezel2004-Intermittent}
S{\'e}bastien Gou{\"e}zel.
\newblock Central limit theorem and stable laws for intermittent maps.
\newblock {\em Probability Theory and Related Fields}, 128(1):82--122, 2004.

\bibitem{Gouezel2004}
S\'ebastien Gou\"ezel.
\newblock Sharp polynomial estimates for the decay of correlations.
\newblock {\em Israel Journal of Mathematics}, 139(1):29--65, 2004.

\bibitem{Gouezel2005}
S{\'e}bastien Gou{\"e}zel.
\newblock Berry--esseen theorem and local limit theorem for non uniformly
  expanding maps.
\newblock In {\em Annales de l'IHP Probabilit{\'e}s et Statistiques},
  volume~41, pages 997--1024, 2005.

\bibitem{Hennion-2001}
Hubert Hennion and Lo{\"{\i}}c Herv{\'e}.
\newblock {\em Limit theorems for {M}arkov chains and stochastic properties of
  dynamical systems by quasi-compactness}, volume 1766 of {\em Lecture Notes in
  Mathematics}.
\newblock Springer-Verlag, Berlin, 2001.

\bibitem{Liverani1996}
Carlangelo Liverani.
\newblock Central limit theorem for deterministic systems.
\newblock {\em International Conference on Dynamical Systems
  (Montevideo,1995)}, 362:56--75, 1996.

\bibitem{Liverani-2004}
Carlangelo Liverani.
\newblock Invariant measures and their properties. a functional analytic point
  of view.
\newblock {\em Dynamical systems. Part II}, pages 185--237, 2004.

\bibitem{Liverani&Terhesiu2015}
Carlangelo Liverani and Dalia Terhesiu.
\newblock Mixing for some non-uniformly hyperbolic systems.
\newblock {\em Annales Henri Poincaré}, pages 1--48, 2015.

\bibitem{Pianigiani1980}
Giulio Pianigiani.
\newblock First return map and invariant measures.
\newblock {\em Israel Journal of Mathematics}, 35(1-2):32--48, 1980.

\bibitem{Rams2003}
Michał Rams.
\newblock Absolute continuity of the sbr measure for non-linear fat baker maps.
\newblock {\em Nonlinearity}, 16(5):1649, 2003.

\bibitem{Robinson1999}
Clark Robinson.
\newblock Dynamical systems: stability, symbolic dynamics.
\newblock {\em Chaos, CRC Press}, 1999.

\bibitem{Sarig2002}
Omri Sarig.
\newblock Subexponential decay of correlations.
\newblock {\em Inventiones mathematicae}, 150(3):629--653, 2002.

\bibitem{Tsujii2001}
Masato Tsujii.
\newblock Fat solenoidal attractors.
\newblock {\em Nonlinearity}, 14(5):1011, 2001.

\bibitem{Young1998}
Lai-Sang Young.
\newblock Statistical properties of dynamical systems with some hyperbolicity.
\newblock {\em Annals of Mathematics}, 147(3):pp. 585--650, 1998.

\end{thebibliography}
\end{document}